\documentclass[12pt]{amsart}
\title[Catalan numbers: from FC elements to classical diagram algebras]{Catalan numbers: from FC elements to classical diagram algebras}
\author{Sadek AL HARBAT}
\address{} 
\email{sadekharbat@inst-mat.utalca.cl}
\usepackage{epigraph}
\usepackage{etex} %bug fixing
\usepackage[latin1]{inputenc}	% widows
\usepackage[T1]{fontenc}
\usepackage[english]{babel}
\usepackage{amsmath,amssymb}
\usepackage{wasysym}
\usepackage{stmaryrd}
\usepackage[table]{xcolor}
\usepackage{color}
\usepackage{dsfont}\let\mathbb\mathds
\usepackage[all]{xy}
\usepackage{graphicx}
\usepackage{mathenv}
\usepackage{lastpage}
\usepackage{fancyhdr}	%réaliser vous-même vos en-têtes et pieds de page relativement facilement.
\usepackage{subfigure}
\usepackage{geometry}
\usepackage[version=3]{mhchem}
\usepackage[force,almostfull]{textcomp}
\usepackage{array}
\usepackage{lipsum}		% pour mettre du texte aléatoire via \lipsum
\usepackage{eso-pic}
\usepackage{multirow}
\usepackage{fancybox}
\usepackage{cite}
\usepackage{amsthm}
\usepackage{listings}
\usepackage{lscape}
\usepackage{multicol}		% Pour pouvoir faire plusieurs colonnes
\usepackage{setspace}
\usepackage{times}
\usepackage{eurosym}
\usepackage{latexsym}
\usepackage{ulem}
\usepackage{lmodern}
\usepackage{colortbl}
\usepackage{rotating}
\usepackage{type1cm}
\usepackage{lettrine}
\usepackage{ragged2e}
\usepackage{mathrsfs}
\usepackage{scrtime}
\usepackage{enumitem}

\usepackage{caption}
\usepackage{longtable}
\usepackage{url}
\usepackage{nomentbl}
\usepackage{nomencl}
\usepackage{hyperref}
\hypersetup{
pdfpagemode=UseOutlines,      % UseOutlines, UseThumbs, None, FullScreen : agencement au démarrage
pdfstartview=Fit,             % Fit, FitH, FitB, FitBH : vue de la page au départ (pleine largeur...)
pdffitwindow=true,            % bool: Maximiser
pdfpagelayout=TwoColumnsRight,% SinglePage, TwoColumnsRight/Left, OneColumn : affichage des pages
pdftoolbar=true,              % bool: Affichage de la barre d'outils
pdfmenubar=true,              % bool: Affichage de la barre de menus
bookmarksopen=false,          % bool: Dépliement des signets
bookmarksnumbered=true,       % bool: Numérotation des signets
colorlinks=true,              % bool: Liens colorés
pdfauthor={Ton nom},          % Auteur
pdftitle={Titre PDF},         % Titre
pdfcreator=PDFLaTeX,          % 
pdfproducer=PDFLaTeX,         %
linkcolor=blue,               % Couleur des liens
urlcolor=blue,                % url
anchorcolor=black,            % du texte
citecolor=blue,               % Couleur de citation 
frenchlinks=true,             % bool: Utiliser des petites majuscules pour les liens, plutôt que de la couleur
pdfborder={0 0 0}             % Ne pas encadrer les liens
}
\usepackage{pgf, tikz}
\usetikzlibrary{matrix,positioning}
\usetikzlibrary{arrows,decorations.markings}

\newtheorem{theorem}{Theorem}[section]

\newtheorem{definition}[theorem]{Definition}
\newtheorem{proposition}[theorem]{Proposition}
\newtheorem{lemma}[theorem]{Lemma}
\newtheorem{corollary}[theorem]{Corollary}
\newtheorem{example}[theorem]{Example}
\newtheorem{remark}[theorem]{Remark}
\newtheorem{remarks}[theorem]{Remarks}
\newenvironment{demo}{\begin{proof}}{\end{proof}}

 \def\C{\mathcal C}

\addto\captionsfrenchb{}
\addto\captionsfrenchb{}

    \newlength{\myarrowsize} 
    \newlength{\myoldlinewidth}

    \pgfarrowsdeclare{myto}{myto}{
        \pgfsetdash{}{0pt} 
        \pgfsetbeveljoin 
        \pgfsetroundcap 
        \setlength{\myarrowsize}{0.5pt}
        \addtolength{\myarrowsize}{.5\pgflinewidth}
        \pgfarrowsleftextend{-4\myarrowsize-.5\pgflinewidth} 
        \pgfarrowsrightextend{.7\pgflinewidth}
    }{
        \setlength{\myarrowsize}{0.4pt} 
        \addtolength{\myarrowsize}{.3\pgflinewidth}  
        \setlength{\myoldlinewidth}{\pgflinewidth}
        \pgfsetroundjoin
        \pgfsetlinewidth{0.0001pt}
        \pgfpathmoveto{\pgfpoint{0.43\myarrowsize}{0}}
        \pgfpatharc{0}{70}{0.14\myarrowsize}
        \pgfpatharc{-80}{-169.5}{4\myarrowsize}
        \pgfpatharc{150}{189}{0.95\myarrowsize and 0.95\myarrowsize}
        \pgfpatharc{0}{-40}{15\myarrowsize}
        \pgfpathmoveto{\pgfpoint{0.43\myarrowsize}{0}}
        \pgfpatharc{0}{-70}{0.14\myarrowsize}
        \pgfpatharc{80}{169.5}{4\myarrowsize}
        \pgfpatharc{-150}{-189}{0.95\myarrowsize and 0.95\myarrowsize}
        \pgfpatharc{0}{40}{15\myarrowsize}
        \pgfpathclose
        \pgfsetstrokeopacity{0.25}
        \pgfusepathqfillstroke
    }

    \pgfarrowsdeclare{myonto}{myonto}{
        \pgfsetdash{}{0pt} 
        \pgfsetbeveljoin 
        \pgfsetroundcap 
        \setlength{\myarrowsize}{0.5pt}
        \addtolength{\myarrowsize}{.5\pgflinewidth}
        \pgfarrowsleftextend{-4\myarrowsize-.5\pgflinewidth} 
        \pgfarrowsrightextend{.7\pgflinewidth}
    }{
        \setlength{\myarrowsize}{0.4pt} 
        \addtolength{\myarrowsize}{.3\pgflinewidth}  
        \setlength{\myoldlinewidth}{\pgflinewidth}
        \pgfsetroundjoin
        \pgfsetlinewidth{0.0001pt}
        \pgfpathmoveto{\pgfpoint{0.43\myarrowsize}{0}}
        \pgfpatharc{0}{70}{0.14\myarrowsize}
        \pgfpatharc{-80}{-169.5}{4\myarrowsize}
        \pgfpatharc{150}{189}{0.95\myarrowsize and 0.95\myarrowsize}
        \pgfpatharc{0}{-40}{15\myarrowsize}
        \pgfpathmoveto{\pgfpoint{-7\myarrowsize}{0}}
				\pgfpatharc{0}{70}{0.14\myarrowsize}
        \pgfpatharc{-80}{-169.5}{4\myarrowsize}
        \pgfpatharc{150}{189}{0.95\myarrowsize and 0.95\myarrowsize}
        \pgfpatharc{0}{-40}{15\myarrowsize}
        \pgfpathmoveto{\pgfpoint{0.43\myarrowsize}{0}}
        \pgfpatharc{0}{-70}{0.14\myarrowsize}
        \pgfpatharc{80}{169.5}{4\myarrowsize}
        \pgfpatharc{-150}{-189}{0.95\myarrowsize and 0.95\myarrowsize}
        \pgfpatharc{0}{40}{15\myarrowsize}
			  \pgfpathmoveto{\pgfpoint{-7\myarrowsize}{0}}
        \pgfpatharc{0}{-70}{0.14\myarrowsize}
        \pgfpatharc{80}{169.5}{4\myarrowsize}
        \pgfpatharc{-150}{-189}{0.95\myarrowsize and 0.95\myarrowsize}
        \pgfpatharc{0}{40}{15\myarrowsize}
        \pgfpathclose
        \pgfsetstrokeopacity{0.25}
        \pgfusepathqfillstroke
        \pgfusepathqfillstroke
    }

 \pgfarrowsdeclare{myhook}{myhook}{
        \setlength{\myarrowsize}{0.6pt}
        \addtolength{\myarrowsize}{.5\pgflinewidth}
        \pgfarrowsleftextend{-4\myarrowsize-.5\pgflinewidth} 
        \pgfarrowsrightextend{.7\pgflinewidth}
    }{
        \setlength{\myarrowsize}{0.6pt} 
        \addtolength{\myarrowsize}{.5\pgflinewidth}  
        \pgfsetdash{}{+0pt}
        \pgfsetroundcap
        \pgfpathmoveto{\pgfqpoint{-2pt}{-6\pgflinewidth}}
        \pgfpathcurveto
            {\pgfqpoint{4\pgflinewidth}{-4.667\pgflinewidth}}
            {\pgfqpoint{4\pgflinewidth}{0pt}}
            {\pgfpointorigin}
        \pgfusepathqstroke
    }

		\pgfarrowsdeclare{my to}{my to}
{
  \pgfarrowsleftextend{-2\pgflinewidth}
  \pgfarrowsrightextend{\pgflinewidth}
}
{
  \pgfsetlinewidth{0.8\pgflinewidth}
  \pgfsetdash{}{0pt}
  \pgfsetroundcap
  \pgfsetroundjoin
  \pgfpathmoveto{\pgfpoint{-5.5\pgflinewidth}{7.5\pgflinewidth}}
  \pgfpathcurveto
  {\pgfpoint{-4.0\pgflinewidth}{0.1\pgflinewidth}}
  {\pgfpoint{0pt}{0.25\pgflinewidth}}
  {\pgfpoint{0.75\pgflinewidth}{0pt}}
  \pgfpathcurveto
  {\pgfpoint{0pt}{-0.25\pgflinewidth}}
  {\pgfpoint{-4.0\pgflinewidth}{-0.1\pgflinewidth}}
  {\pgfpoint{-5.5\pgflinewidth}{-7.5\pgflinewidth}}
  \pgfusepathqstroke
}

		\pgfarrowsdeclare{mytodashed}{mytodashed}
{
 \pgfarrowsleftextend{-2\pgflinewidth}
 \pgfarrowsrightextend{\pgflinewidth}
}
{
 \pgfsetlinewidth{0.8\pgflinewidth}
  \pgfsetdash{{0.8cm}{0.8cm}}{0cm}
\pgfsetroundcap
\pgfsetroundjoin
  \pgfpathmoveto{\pgfpoint{-5.5\pgflinewidth}{7.5\pgflinewidth}}
\pgfpathcurveto
{\pgfpoint{-4.0\pgflinewidth}{0.1\pgflinewidth}}
 {\pgfpoint{0pt}{0.25\pgflinewidth}}
{\pgfpoint{0.75\pgflinewidth}{0pt}}
\pgfpathcurveto
{\pgfpoint{0pt}{-0.25\pgflinewidth}}
  {\pgfpoint{-4.0\pgflinewidth}{-0.1\pgflinewidth}}
 {\pgfpoint{-5.5\pgflinewidth}{-7.5\pgflinewidth}}
\pgfusepathqstroke
}

\tikzstyle{vecArrow} = [thick, decoration={markings,mark=at position
   1 with {\arrow[semithick]{open triangle 60}}},
   double distance=1.4pt, shorten >= 5.5pt,
   preaction = {decorate},
   postaction = {draw,line width=1.4pt, white,shorten >= 4.5pt}]
\tikzstyle{innerWhite} = [semithick, white,line width=1.4pt, shorten >= 4.5pt]

\newcommand\dyckpath[3]{

	\draw[help lines] (#1) grid +(#2,#2);
	\draw[dashed] (#1) -- +(#2,#2);
	\coordinate (prev) at (#1);
	\foreach \dir in {#3}{
		\ifnum\dir=0
		\coordinate (dep) at (1,0);
		\else
		\coordinate (dep) at (0,1);
		\fi
		\draw[line width=2pt,] (prev) -- ++(dep) coordinate (prev);
	};
}

\newcommand\dyckpathballot[3]{\small 

	\draw[help lines] (#1) grid +(#2,#2);
	\draw[line width=0.2pt,dashed] (#1) -- +(#2,#2);
	\coordinate (prev) at (#1);
	\foreach \dir in {#3}{
		\ifnum\dir=0
		\coordinate (dep) at (1,0);
           \draw[line width=2pt,] (prev) -- ++(dep)  node at ++(-0.6,0.3) {\tiny +} coordinate (prev);
		\else
		\coordinate (dep) at (0,1);
    \draw[line width=2pt,] (prev) -- ++(dep)  node at ++(-0.3,-0.4){-} coordinate (prev);
		\fi
	};
}

	\makeatletter
	\newcommand\POSITION[3]{%
	\begingroup
	\@tempdim@x=0cm
	\@tempdim@y=\paperheight
	\advance\@tempdim@x#1
	\advance\@tempdim@y-#2
	\put(\LenToUnit{\@tempdim@x},\LenToUnit{\@tempdim@y}){#3}%
	\endgroup
	}
\makeatother

\addto\captionsenglish{}

\begin{document}

	\begin{abstract}
	Let $W^c(A_n)$ be the set of fully commutative elements in the $A_n$-type Coxeter group. Using only the settings of their canonical form, we recount $W^c(A_n)$ by the recurrence that is taken as a definition of the Catalan number $C_{n+1}$ and we find the Narayana numbers as well as the Catalan triangle via suitable set partitions of $W^c(A_n)$. We determine the unique bijection between $W^c(A_n)$ and the set of non-crossing diagrams of $n+1$ strings that respects the  diagrammatic multiplication by concatenation in the $A_n$-type Temperley-Lieb algebra, along with the two algorithms implementing this bijection and its inverse. 
	
	\end{abstract}

 	\maketitle 
	
	\keywords{Catalan number; Coxeter group; Fully commutative element; Diagram algebra.}

\section{Introduction}

Following R. P. Stanley \cite[\S 1.2]{Sta15}  the recurrence and initial condition
\begin{equation}\label{veryfirst}
 C_{n+1} = \sum_{k=0}^{n}   C_{k}C_{n-k}, ~~~~ C_{0}=1,
\end{equation}
is "the most important and most transparent recurrence satisfied by" the Catalan numbers $1,1, 2, 5, 14 \dots$, we hence adopt it as a definition. In order to avoid (and how one could !) drowning in the History of those numbers  we refer the reader to \cite{Sta15}, of which we may (and should) borrow some terminology. Starting now: We refer to any interpretation by its number in the book in double brackets. 

One of the aims of this work is to add "explicitly" a 215th object counted by Catalan "by definition". Citing again \cite{Sta15}, "For other `Catalan objects', however, it can be quite difficult, if not almost impossible, to see directly why the recurrence (\ref{veryfirst}) holds." So saying that we count our object  (the set of fully commutative elements in the Coxeter group of type $A_n$) by definition is saying exactly that this defining recurrence holds in this object, via suitable set partitions. 

Another aim is to give some bijections with other objects, that are of specific importance for algebraists and topologists at the same time. 

We focus in this work on two Catalan phenomena (say counted by  $C_{n+1}$): 

\begin{itemize}
\item  Fully commutative elements in the Coxeter group of type $A_n$;

\item Non-crossing diagrams with $n+1$ strings ((61)).
\end{itemize}

And other two phenomena, in general better known and considered two of the five fundamental interpretations of Catalan numbers in  \cite{Sta15}, of which we explain roughly the relation with the two first: 

\begin{itemize}

\item Dyck paths of length $2(n+1)$ ((24));

\item Ballot sequences of length  $2(n+1)$ ((77)).

\end{itemize}

We name them by their shortcuts: $F, N, D$ and $B$ respectively, in order to shorten the notation of bijections among them, in particular the ones  already known. 

Let $(W,S)$  be a Coxeter system. We say that w in  W is {\it Fully Commutative} (otherwise {\it FC}) if any reduced expression for $ w$
can be obtained from any other using only commutation relations among the members of $S$.
If $W$ is simply laced then the FC elements of $W$ are those with no $sts$  factor in any 
reduced expression, where $t$ and $s$ are any non-commuting generators. This work is centered on  $W^c(A_n)$: The set of FC elements in $W(A_n)$, the Coxeter group of type $A_n$  (AKA $Sym_{n+1} $). The careful reader should not be confused that we deal with the $n+1$-th Catalan number $C_{n+1}$ rather than $C_n$ since we prefer to treat $W (A_n)$ (not $W (A_{n-1})$).  \\

Our starting point is a canonical reduced expression of an FC element, say $w$, in $W^c(A_n)$, this canonical expression is well known and explained in section  \ref{notations}, it consists essentially of well defining a canonical form for $w$ (hence well defining $w$) by $p$ pairs of positive integers  $(i_t , j_t)  $ for  $1\le t \le p $,   such that $ i_t    \le j_t $ and 
$$
\begin{aligned}
 ~~ ~~~~~&n\ge j_1 > \dots > j_p \ge 1 \text{, and  } \\
&n\ge i_1 > \dots > i_p \ge 1.
\end{aligned}
$$
We stress the fact that we look at  such a $w$ as a  \underline{$p$-pairs $(i_t , j_t)  $}, for although the set $W^c(A_n)$ coincides with the set of $321$-avoiding permutations in the symmetric group $Sym_{n+1}$   \cite{Sta15} ((126)) when looking at the latter as $W(A_n)$,  yet the permutations are different combinatorial objects than FC elements given by their Coxeter generators expressions, specifically as a $p$-pairs: the center of this paper.  \\

{\bf Bijections}: It is in  \cite{Jones_1983} that this canonical reduced expression was born, along with a bijection with Dyck paths by sending every pair to a peak (say by the bijection $FD$). Now paths are in bijection with ballots by sending every vertical step to $(+)$ and every horizontal step to $(-)$ (say by $DB$).  Those bijections and their inverses can be easily seen in the following example.  
\begin{example}\label{ex5}  (for $n=5$) $ (4,5) (3,3) (1,1)$ (the three-pairs defining $  \sigma_4 \sigma_5 \ \sigma_3 \ \sigma_1 $). By $FD$ the  corresponding Dyck path has 
three peaks with 
coordinates   $ (1,1) (3,3) (5,4)$ (bottom to top) and bottom to top is indeed the direction of the corresponding ballot  $+-++--++-+--$ (starting always by $+$):\\

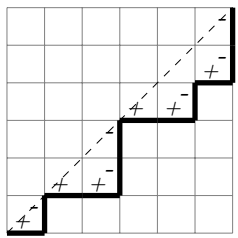
\centerline{
			\begin{tikzpicture}[scale=0.5]
\dyckpathballot{0,0}{6}{0,1,0,0,1,1,0,0,1,0,1,1}; 
\end{tikzpicture}}										 
\end{example}
Dyck paths were known to be counted by Catalan numbers, so does any set in bijection with them, $W^c(A_n)$ in particular. In the first part of this work, we  count FC elements by definition, using only the settings of their canonical form without the help of any bijection.
This relies on set-partitions of FC elements, 
%hence phenomena in bijection with $W^c(A_n)$,
 in particular the notion of "thick" and "thin" elements, see Definition \ref{thin}. Hence counting $W^c(A_n)$ by definition puts FC elements in the class of "basic" interpretations of Catalan (the class being defined as phenomena counted by the definition of $C_{n+1}$ and not only by bijections). $C_{n+1}$ has two famous partitions: Narayana numbers $C^p_n$ and Catalan triangle numbers $^i S_n$.  Since we fix $n$, we choose to call each of them a one-parameter partition (the parameters being $p$,$i$ respectively). \\

{\bf Partitions}: Call $p$ the {\it size} of $w$, it is natural to ask: What is the number of FC elements of size $p$? And what is the number of FC elements such that  $i_1=i$ ? It turns out  (and again only by the FC settings and its set-partitions) that  the answers are $C^p_n$ and  $^i S_n$ in this order. We quote:  "It is interesting to find for each of the combinatorial interpretations of $C_{n+1}$ a corresponding decomposition into subsets counted by Narayana" (\cite{Sta15} p.125), and we may add subsets counted by Catalan triangle. We study the mixed: Narayana-triangle two-parameters partition, and other partitions mostly all known for other phenomena like paths and ballots, but which did not have much meaning in "Diagrams" (the center of the last chapter), at least not before the bijection that we build in Theorem \ref{FC D}. Explaining those partitions and their ramifications led to a wider view on  $W^c(A_n)$ itself! See for example the duality in 
Proposition \ref{duality}. \\

{\bf Non-crossing diagrams}: One of the most wanted object, chased at the same time in algebra as in low dimensional topology, knot theory in particular, is:  Non-crossing diagrams. These classic diagrams form a basis for the well known Temperley-Lieb algebra, which has another basis indexed by FC elements: $\{ e_w; w\in W^c(A_n) \} $. Let's call them FC monomials. The FC monomial  $e_w$ is independent of the reduced expression of $w$, let's pick up the canonical one (see Theorem \ref{1_2}), we have: 

$e_{w} = e_{i_1}e_{i_1+1} .. e_{j_1} \dots e_{i_p}e_{i_p+1} .. e_{j_p}. $

Now it is to be understood that algebraically $e_w$ corresponds to a unique diagram $D_w$ coming from the concatenation of $\ell(w)$ primitive diagrams: $ D_{i_1}$ on the top of $ D_{i_1+1} \dots $ on the top of $ D_{j_p}$. 
It is well known that ballot sequences of length $2(n+1)$  are in bijection with non-crossing diagrams with $n+1$ strings and with FC elements  in $W(A_n)$, so one can determine  a bijection $DF$ from Diagrams to FC elements  (hence FC monomials) directly by composing those bijections! Yet it is not "the bijection" which is compatible with the multiplication by concatenation of diagrams, see the counterexample in \ref{cex}. 
The main result of this work is the explicit description of the bijection between $W^c(A_n)$ and non-crossing diagrams with $n+1$ strings that is compatible with the multiplication of monomials (hence with diagram concatenation) i.e. that sends $w$ to $D_w$, this is Theorem \ref{FC D}. 

This bijection yields two algorithms: One  permits to "draw" into a diagram any FC monomial directly without any concatenation or multiplication after putting $w$ in its canonical expression, and the other receives any diagram as an input and gives directly the FC monomial (hence FC element) in its canonical form without intermediate  objects or computation. Concerning the first, we do not know of another existing algorithm, other than the one coming from concatenating diagrams. While concerning the second algorithm, as far as we know the fastest algorithm is  one that engages at least standard tableaux and walks as intermediate objects, see for details \cite{Steen}. 

We finish by proposing some applications and open questions on which the reader is warmly welcomed to think. We call attention to Proposition \ref{LwRw} that expresses the left and right descent sets of a FC element $w$ in terms of the corresponding diagram, in a very simple way. Some consequences on cellularity properties are to be expected.  \\

{\bf A bit farther}: One would like -as a next but not far aim-  to "fullcommutative" any interpretation of Catalan, and see the meaning of the set-partitions (hence partitions) that we propose, the author would like specifically to contaminate two Catalanomena:

\begin{itemize}

\item Triangulations of a convex polygon with $n+2$ vertices ((1));

\item Coxeter friezes with $n+2$ rows ((197)).

\end{itemize} 
 
 We finish this introduction by pointing out that this work is followed by a second one in which we treat the type $B$  FC elements, hence giving the suitable bijection between the "positive" $B$ FC elements and blobbed diagrams, and some generalization with symplectic blob algebra (see \cite{Sadek_David}  for the notion of positivity).

\section{ Canonical form and basic bijections}

\subsection{A canonical form for FC elements}\label{notations}

		 	Let $(W,S)$ be a Coxeter system with associated Coxeter graph $\Gamma$. For $s, t$ in $S$ we let $m_{st}$ be the order of $st$ in $W$. Let $w\in W$. We denote by $\ell(w)$ the length of a $w$ with respect to $S$. We call 
			{\it support of $w$} and  denote by $Supp(w)$ the subset of $S$ consisting of all generators appearing in a (any) reduced expression of $w$.   We define $\mathscr{L} (w) $ to be the set of $s\in S$ such that $\ell(sw)<\ell(w)$, in other words  $s$ appears at the left edge of some reduced expression of $w$. We define $\mathscr{R}(w)$ similarly.

				We know that from a given reduced expression of $w$ we can arrive to any other reduced expression only by applying braid relations \cite[\S 1.5 Proposition 5]{Bourbaki_1981}. Among these relations there are commutation relations: those 
				that  correspond  to   generators $t$ and $s$ with $m_{st} = 2$.

	             \begin{definition}
			Elements for which we can pass from any reduced expression to any other only by applying commutation relations are called {\rm fully commutative elements} (from now on  {\rm FC elements}). We denote  by $W^{c}$ the set of FC elements in $W$. 
		    \end{definition} 
					
			Consider the $A$-type Coxeter group with $n$ generators $W(A_{n})$, ($n$ being a positive integer and $W(A_{0})= 1$) with the following Coxeter-Dynkin graph:

			
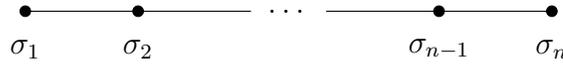
\begin{figure}[ht]
				\centering
				 
				\begin{tikzpicture}

  \filldraw (0,0) circle (2pt);
  \node at (0,-0.5) {$\sigma_{1}$}; 
   
  \draw (0,0) -- (1.5, 0);

  \filldraw (1.5,0) circle (2pt);
  \node at (1.5,-0.5) {$\sigma_{2}$};

  \draw (1.5,0) -- (3, 0);

  \node at (3.5,0) {$\dots$};

  \draw (4,0) -- (5.5, 0);
  
  \filldraw (5.5,0) circle (2pt);
  \node at (5.5,-0.5) {$\sigma_{n-1}$};
 
  \draw (5.5,0) -- (7, 0);
  
  \filldraw (7,0) circle (2pt);
  \node at (7,-0.5) {$\sigma_{n}$};

               \end{tikzpicture}
			 
			\end{figure}					
							
$$
\begin{aligned}
\text{ We let:   } ~~~~~~  [ i,j ] &= \sigma_i \sigma_{i+1} \dots \sigma_j   \   \text{ for }  1\le i\le j \le n.  % \  \text{ and } \    [ 0,1 ] = 1, 
\\
%[  i,j ]  &= \sigma_i \sigma_{i+1} \dots \sigma_j   \    \text{ for } 1\le i\le j \le n \    \text{ and }  \   [  n+1, n ]  = 1, 
\\
%\qquad  \quad    h(i,r) & =   [ i,1 ] [  r,n ] 
%  \quad   \text{ for } 0\le i < r \le n+1  \text{ and } (i,r)  \ne (0,1),  
  \end{aligned}
$$ 
 \medskip
Considering right classes of $W(A_{n-1})$ in $W(A_{n}) $, Stembridge has described canonical reduced words for   elements of 
$W(A_n)$. We get,  setting $W^c(A_n)= A^c_n$:

\begin{theorem}\label{1_2}{\rm \cite[Corollary 5.8]{St}}
 Let $n$ be a positive integer, then   $A^c_n$ is the set of elements of the   form: 
 \begin{equation}\label{eq:Stembridge}
  [i_1, j_1]  [ i_2, j_2 ]  \dots   [i_p, j_p] , 
 \text{ with } 0\le p \le n  \text{ and } 
 \left\{ \begin{matrix}
 n\ge j_1 > \dots > j_p \ge 1 ,  \cr    n\ge i_1 > \dots > i_p \ge 1,   \cr
 j_t \ge i_t   \text{ for }  1\le t \le p.  
 \end{matrix}\right.
 \end{equation}
We call {\em size} of $w$ the integer $p$ ( with $0$ as the size of the identity). We call such pairs  {\em  standard $p$-pairs}. 
\end{theorem}

Inspecting the inequalities  above, we see that  the only term in expression (\ref{eq:Stembridge}) 
in which  $ \sigma_{n} $ 
  can   occur  is the first term. If $ \sigma_{n} $ does occur, then 
$j_1$ must be equal to $n$ and, whether or not $j_1$ is equal to  $n$, $\sigma_{n} $  occurs only once at most.   Similarly,  if $\sigma_{1} $ does occur in expression (1), 
  then $i_p=1$, and $\sigma_{1} =\sigma_{i_{p}}$   appears only once. 

We let $\iota$ be the automorphism of the Coxeter graph given by 
$\iota(j)= n-j+1$ for $1\le j \le n$,  and the corresponding automorphism of $W(A_{n}) $ 
given on the generators by $\iota(\sigma_j)=\sigma_{\iota(j)}$. 
We observe that composing $\iota$ with the inverse map provides an  anti-involution  
$\Delta$ of 
  $W(A_{n}) $ that preserves canonical forms. 
  
  This very canonical form was established in \cite{Jones_1983} using the monomial-diagrammatical generators of the Temperley-Lieb algebra  (the latter is the center of the second part  of this work) rather than the Coxeter generators! Before even the birth of the notion of full commutativity in Coxeter groups. From now on we mean by  a  FC element, a  FC  element in $W(A_n)$.

\subsection{Catalania: Dyck paths and ballots}

Dyck paths of length $2(n+1)$  are lattice paths  from $(0,0)$ to $(n+1, n+1)$ with steps $(0,1)$ and $(1,0)$, such that the path never rises above the line $x=y$ ((24)).

A ballot sequence of length $2(n+1)$ is a sequence with $n+1$ each of $1$'s and $-1$'s in a way in which every partial sum is nonnegative ((77)) left to right. There are two obvious bijections between  Dyck paths and ballot sequences (of length $2(n+1)$ for each). The one sending the step  $(1,0)$ to 1 and  $(0,1)$ to $-1$ and reading from bottom to top (let's call it $DB$). The other sending $(1,0)$ to $-1$ and reading from top to bottom. We write sequences with $1$ (resp. $-1$) denoted as $+$ (resp. $-$). Obviously $DB$ starts from $(0,0)$. The inverse bijection $BD$ is exactly as one would expect. 

In  \cite{Jones_1983} the late V. Jones has determined a bijection between Dyck paths and standard $p$-pairs (mentioned as an observation of H. Wilf), call it $DF$, the bijection sending any path (with $p$ peaks having  $(j_t, i_t)$ as coordinates) to the $p$-pairs $(i_t, j_t)$, thus to the canonical form in (\ref{eq:Stembridge}), since those pairs are standard. Hence, he showed that {\em The Catalan numbers count FC elements} via $DF$. In the next chapter we show that Catalan numbers count those elements "by definition", i.e. without the help of  Dyck paths, and we study the consequences of such interpretation. By composing  $DB$ after $FD$ we get $FB$ (which is treated in details in the last chapter and given explicitly in \ref{cex}). Nevertheless, in Example \ref{ex5} we can see the simplicity of presenting $FB$  by inserting Dyck paths, this insertion explains three bijections (with their inverses) in only one figure! 
 
   \section{New set-partitions: Catalan numbers coming directly from FC elements}

\subsection{ FC elements counted by the inductive definition of Catalan numbers.}

Now take any element $w$ in $A^c_n$:
  $$
  w= [ i_1, j_1 ] [ i_2, j_2 ] \dots  [ i_p, j_p ] ,
  $$
  
     \begin{definition}
			Suppose that $p>0$. We say that $w$ is thick if $j_t > i_t   \text{ for }  1\le t \le p$. We call $w$ slim if it is not thick. The set of elements of size $p$ in $A^c_n$ is denoted $A^{c,p}_{n}$.
		    \end{definition}\label{thin}
		    Set  $A^{G}_{n}$ (resp, $A^{S}_{n}$) the set of thick (resp, slim) elements in $A^c_n$. Moreover we call $A^{G,p}_{n}$ (resp, $A^{S,p}_{n}$) the set of $p$-size thick  (resp,  $p$-size slim) elements in $A^c_n$. For simplicity (and almost by definition) whenever the index $n$ or $p$ in the last sets is smaller or equal to $0$, then the set  is the empty set. We have clearly the following set-partitions : 
	$$
	A^c_n =~~ \{1 \}\bigsqcup A^{G}_{n} \bigsqcup A^{S}_{n} =~~ \{1 \}  \sqcup \bigcup_{p=1}^{p=n} A^{G,p}_{n}  \sqcup \bigcup_{p=1}^{p=n} A^{S,p}_{n}.
	$$
	\begin{lemma}\label{bijectionthick}
	
	Let  $w= [ i_1, j_1 ] [ i_2, j_2 ] \dots  [ i_p, j_p ] $ be in $A^{G,p}_{n}$, then $\bar{w}$ defined by its canonical form 
$$ \bar{w}=  [ i_1, j_1 -1 ] [ i_2, j_2 -1 ] \dots  [ i_p, j_p -1 ] $$ is in $A^{c,p}_{n-1}$ and we have the bijection: 	 
	   \begin{eqnarray}
					         \bar{.}: A^{G,p}_{n} &\longrightarrow& A^{c,p}_{n-1} \nonumber\\
					         w &\longmapsto&  \bar{w} $ ~~~ $ \nonumber
					%a_{n} &\longmapsto& \sigma_{n} a_{n+1}\sigma^{-1}_{n}. \nonumber\\\nonumber
				\end{eqnarray}
				in such a way that $\ell(w) = \ell(\bar{w})+p$. This obviously gives the following size preserving bijection: 
				   \begin{eqnarray}
					        \bar{.}: A^{G}_{n} &\longrightarrow&A^c_{n-1} -\left \{ 1 \right \} \nonumber\\
					         w &\longmapsto&  \bar{w} $ ~~~ $ \nonumber
					%a_{n} &\longmapsto& \sigma_{n} a_{n+1}\sigma^{-1}_{n}. \nonumber\\\nonumber
				\end{eqnarray}

	\end{lemma}

	\begin{demo}
	Clear.
	\end{demo}
	
	\begin{remark}
	
	Actually the interpretation ((107)) in  \cite{Sta15} expresses indeed the thick elements of $W^c(A_n)$ along with the identity:   $A^{G}_{n} \cup \left \{ 1 \right \}$, which we have just seen are in bijection with $W^c(A_{n-1})  $ ! With cardinality $C_n$. This thick-elements interpretation is the closest one in \cite{Sta15} to $W^c(A_n)$ as  standard $p$-pairs as looked at in this work. 
	
	\end{remark}

	The following observation is a key step of this work: take an element $w$ in  $A^{S,p}_{n}$: 
	$$
	[ i_1, j_1 ] [ i_2, j_2 ] \dots  [ i_p, j_p ]. 
	$$
	Suppose that the first block from the left made of one simple reflection is the $s$-term, in other words: 
	$$
	w=\underbrace{[ i_1, j_1 ]    \dots [ i_{s-1},  j_{s-1} ] }_{X}  [ i_{s} = j_{s},  j_{s} ]  \underbrace{[ i_{s+1}, j_{s+1} ]  \dots  [ i_p, j_p ] }_{Y}
	$$
where $X$ is either $1$ or a thick element of size $s-1$ supported in $ \left \{ \sigma_{n}, \sigma_{n-1}, \dots \sigma_{s+1} \right \} $, and $Y$ is in  $A^c_{s-1}$. Now from the canonical form conditions on indexes we see that $ i_s < i_{s-1}$ hence  $ i_s < i_t  \le j_t$ for all $1\le t \le s-1$ and $ j_s > j_{s+1}$  hence  $ j_s > j_t  \ge  i_t$ for all $s+1\le t \le p$. Which means that $ \sigma_{i_s} =  \sigma_{j_s}$ separates $X$ and $Y$, in other terms: they are independent from each other for any given $w$ as above. \\
        
More formally, 
 for a slim element $  [ i_1,j_1 ] \cdots [ i_p,j_p ]$ 
we let $r= \min\{t/ 1\le t \le p \text{ and } i_t=j_t \}$. We have $1\le r \le p$ and we let $i=i_r$, $1 \le i \le n$. 
A slim element with this $r$ and this $i$  can be written as $g \sigma_i d$ where $g$ is a thick element of size $r-1$ in $W(\sigma_n, \cdots, \sigma_{i+1} )$ if $r>1$, or $1$ if $r=1$,  and $d$ is any FC element of size $p-r$ in $W(A_{i-1})$. 
The case $i=n$, hence $r=1$,  requires a special treatment:  we get   the product of $\sigma_n$ by a FC element of $W(A_{n-1})$ of size $p-1$. Changing $i$ to $n-i$ and using the isomorphism 
$\tau_{n-i}: W(\sigma_1, \cdots, \sigma_{i} ) \longrightarrow W(\sigma_{n-i+1}, \cdots, \sigma_{n} ) $ given by $\sigma_k \mapsto \sigma_{n-i+k}$, we obtain   the following:

     \begin{proposition} The set-partition $A^c_n =  \{1 \}\cup A^{G}_{n} \cup A^{S}_{n}$  gives us  (schematically): 
 
   \begin{equation}\label{part1}
    A^c_n=  \{1 \} \cup A^{G}_{n} \cup (\bigcup_{i=0}^{i=n-1} (\tau_{n-i}(A^{G}_{i})\cup 1) \  . \ \sigma_{n-i} \  . \ A^{c}_{n-(i+1)}), 
     \end{equation}

      \end{proposition}   
 \begin{demo}
	Clear.
	\end{demo}
The initial condition $C_0=1$ gives that the recurrence relation defining Catalan numbers can be written
$C_{n+1} = C_{n} + \sum_{i=0}^{n-1} C_i C_{n-i}$ for $n\ge 0$. We observe that the previous Proposition and Lemma imply: 
$$
\sharp(A_n^c)= 1+ ( \sharp(A_{n-1}^c )- 1 )+  \sum_{i=0}^{n-1} \sharp (A^{c}_{i-1}) .  \sharp (A^{c}_{n-(i+1)} ) 
$$
which we recognize as the recurrence relation for $C_{n+1}$ above, provided that we make the convention $ \sharp (A^{c}_{-1})=1$. 

  \begin{corollary}       
    $\sharp A^{c}_{n} = C_{n+1}$ by definition.
   \end{corollary}      
         	 This partition of the Catalan number can be interpreted    in the context of Dyck paths as distinguishing between paths touching $x=y$ (slim paths) and paths strictly below $x=y$, these can be looked at as paths from $(0,0)$ to $(n,n)$, and the rest is obvious. Yet one can see how natural this separation is in the context of standard  $p$-pairs, i.e. FC elements. 
	 
 \subsection{Sub-partitions (one parameter partitions of $C_{n+1}$): Narayana numbers }

The set of FC elements is naturally partitioned according to the size $p$, $0 \le p \le n$, of its elements  
  (where $1$ is the unique element of size $0$). We  let 
  $\C_n^p$,  $0 \le p \le n$, be the cardinality of $A^{c,p}_{n}$, i.e. the number of FC elements having  size $p$ in $W(A_n)$. By convention we agree on   $\C_n^p=0$ for $n<0$ or $p <0$ or $p > n$. We have: 
  $
  C_{n+1} = \sharp A_n^c = \sum_{p=0}^{p=n}   \C_n^p.
  $
 Our purpose in this section  is to exhibit an explicit formula for the terms $\C_n^p$ using only the settings of FC elements. By definition we have $$\C_{n}^0=1  ;  \quad \C_{n}^n=1. $$ 
For the only element of size $n$ is $\sigma_n \sigma_{n-1} \dots \sigma_2 \sigma_1$. Using the set-partition (\ref{part1}), we  get the formula : 
  \begin{equation}\label{part1cardinalities}
\C_n^p = \C_{n-1}^p +  \C_{n-1}^{p-1} + \sum_{r=1}^p  \sum_{i=1}^{n-1}    \C_{n-i-1}^{r-1}     \   \C_{i-1}^{p-r}
 \qquad (p\ge 1 , n\ge 1). 
     \end{equation}

We notice that for $r=1$ and $i<n$ the terms $\C_{n-i-1}^{r-1}  $ are all $1$ while for  $r>1$   the sum on $i$ is really a sum on $i \le n-r  $ since,   
 for the leftmost thick element, we must have $n \ge j_1 > i_1 > \cdots >i_{r} $;   with our conventions the  terms in excess are $0$. We also observe that the formula holds for $p=0$, as $1=1$. \\  
This implies  inductive formulae for 
 $\C_n^1$, $\C_n^2$,  $\C_n^3$,  from which we can compute directly those quantities, 
using the known values for  the sums of integral powers. We give the first terms as examples: 
$$
\begin{aligned}   \C_{n}^1   &= \C_{n-1}^1 + n  \quad \text{for } n \ge 2, 
\qquad \C_{n}^1 =  1 +  \cdots + n = \frac{n(n+1)}{2}= \frac{1}{2} \binom{n}{1}    \binom{n+1}{1} ; \\ 
     \C_{n}^2 &=     \C_{n-1}^2  
    +   \C_{n-1}^1 +  2   \sum_{i=1}^{n-2} \C_i^1  \quad \text{for } n \ge 3, \quad 
      \C_{n}^2=       \frac{(n-1)n^2 (n+1)}{12} = \frac{1}{3}  \binom{n}{2}     \binom{n+1}{2} ; 
 \\
 \C_{n}^3   &=       \C_{n-1}^3 +   \C_{n-1}^2  + 2 \  \sum_{i=2}^{n-2} \C_{i}^2  
        +   \sum_{i=1}^{n-3} \C_{i}^1 \C_{n-i-2}^1  \quad \text{for } n \ge 4, \\ 
       \C_{n}^3 &=   \frac{(n-2) \ (n-1)^2 \ n^2 \  (n+1) }{144}  =   \frac 1 4 \binom{n}{3} \binom{n+1}{3}       .  
 \end{aligned} $$ 
At that point we conjecture that $\C_{n}^p   =   \frac{1}{p+1} \  \binom{n}{p} \    \binom{n+1}{p}. $

Our aim is to derive from (\ref{part1cardinalities}) an expression for the generating function of 
  $\C_n^p$,  $0 \le p \le n$, namely $E(x,y)= \sum_{n \ge 0, p\ge 0} \C_n^p x^p y^n.$ In the appendix, one can see the details  that drive $E(x,y)$, which is equal to
  $$
 \frac{(1-y-xy)}{2x y^2} \left[  1 - \left(  1 - 4 x y^2 (1-y-xy)^{-2} \right)^{\frac{1}{2}}  \right] = \sum_{t \ge 0}  \frac{1}{t+1} \binom{2t}{t}  x^t y^{2t} (1-y-xy)^{-2t-1}.
$$
 In particular, we can fix $n$ and develop 
the one-variable generating function $F_n(x) = \sum_{p \ge 0} \C_n^p x^p$, which is the coefficient of $y^n$ 
in the above series. We obtain: 

$$\begin{aligned}
F_n(x)   &=     \sum_{p=0}^{n}   \binom{n}{p}    \left(     \sum_{t=0}^{p}     \frac{1}{t+1}  \binom{p}{t} 
 \binom{n-p}{t} 
   \right)   x^{p}.
\end{aligned}
$$

$E(x,y)$ is already known in literature. Actually we chose to use it because we could not check our hypothesis directly using recurrence relation (\ref{part1cardinalities}) which misses the "linearity" that the following recurrence relations in this work have. Summing up, with the simple Lemma \ref{binomiallemma} in the Appendix, we have proved Narayana's theorem ((59))  (in that version  $\C_n^p$ is the number of paths with $p$ peaks): 
\begin{theorem}\label{Narayana}  With the above settings we have
$$
\C_n^p =\frac{1}{p+1} \binom{n }{p}  \binom{n+1}{p}. 
$$
\end{theorem}  
Which implies the Narayana partition of Catalan: $ C_{n+1} =   \sum_{p=0}^{p=n}    \frac{1}{p+1}  \binom{n}{p} \binom{n+1}{p}$  \cite{Narayana}.\\

The fact that 
$ \   \C_{n}^p   $ is equal to $   \dfrac{n! (n+1)!}{p! (n-p)!(p+1)!(n-p+1)!}   $,   implies $  \C_{n}^p =  \C_{n}^{n-p}.$   
We could prove (and going to do in few lines) directly this "duality" property purely by the settings of FC elements and without even needing the value of $  \C_{n}^p $ above, as follows, and in our way we prove a stronger statement, that is to define a  permutation of $A^c_n$ relating length and size,
which we believe  needs more investigation, especially in other Catalan phenomena that are in explicit bijection with FC elements.

Again an element in $A^c_n$ is characterized by  its size $p$, $0\le p \le n$, and  two sequences  $(i_1, \dots, i_p)$ and $(j_1, \dots,j_p)$  satisfying 
$$n\ge i_1 > \dots > i_p \ge 1,   \quad    n\ge j_1 > \dots > j_p \ge 1 ,  \quad    
 j_t \ge i_t   \text{ for }  1\le t \le p. $$ 
Now  $\{ i_1, \dots, i_p\}$ is a subset of $\{  1, \dots, n\}$ with $p$ elements. The complement of this subset 
 is a subset with $n-p$ elements of $\{  1, \dots, n\}$ that can be ordered, hence there is a unique way of writing it  $\{ j_1^\ast, \dots, j_{n-p}^\ast \}$ with  
 $n\ge j_1^\ast > \dots > j_{n-p}^\ast \ge 1 $. In the same way we find a sequence 
 $\{ i_1^\ast, \dots ,i_{n-p}^\ast \}$ satisfying  $n\ge i_1^\ast > \dots > i_{n-p}^\ast \ge 1 $ and 
$\{ i_1^\ast, \dots ,i_{n-p}^\ast \} \cup \{j_1, \dots, j_p\} =  \{  1, \dots, n\} $. We claim that  
$$ j_t^\ast \ge i_t^\ast   \quad    \text{ for }  1\le t \le n- p.$$
Indeed let $a= j_t^\ast$. We want to show that $i_t^\ast \le a$.  The subset $S=\{ n, n-1, \dots, a +1\}$ contains $t-1$ elements of the sequence $(j_s^\ast)$ 
and $t'$ elements of the sequence $(i_s)$, with $t-1+t' =n-a $. Since $j_s \ge i_s$ for all $s$, the subset $S$ contains at least $t'$ elements of the sequence $(j_s)$, hence it contains at most $t-1$ elements of the sequence 
 $(i_s^\ast)$ and   $i_t^\ast \le a$ as claimed. We  subsume this discussion:
   
\begin{proposition}\label{duality}
Let $w=
	[ i_1, j_1 ] [ i_2, j_2 ] \dots  [ i_p, j_p ] 
	 $ be a FC element in $W(A_n)$ given canonically. We denote by $w^\ast$ the "dual" element of $w$, namely 
$$w^\ast =
	[ i_1^\ast, j_1^\ast ] [ i_2^\ast, j_2^\ast ] \dots  [ i_{n-p}^\ast, j_{n-p}^\ast ], 
$$ 
where the strictly decreasing sequences  $\{ i_1^\ast, \dots, i_{n-p}^\ast \}$ and  $\{ j_1^\ast, \dots, j_{n-p}^\ast \}$ are defined by: 
$$\{ i_1^\ast, \dots ,i_{n-p}^\ast \} \cup \{j_1, \dots, j_p\} =  \{ j_1^\ast, \dots ,j_{n-p}^\ast \} \cup \{i_1, \dots, i_p\} =   \{  1, \dots, n\} .$$

Then $w^\ast $ belongs to $A^c_n$ and the duality map $w \mapsto w^\ast$  is an involutive bijection of $A^c_n$ that maps size $p$ elements onto size $n-p$ elements such that $\ell(w) - \ell(w^\ast) = 2p-n$. In particular we have  $$  \C_{n}^p =  \C_{n}^{n-p}.$$
\end{proposition}
As we have just mentioned, this proposition has a strong meaning in the Coxeter settings for it relates the length of any element to the length of its dual, we could not find an equivalent in the literature in other Catalan phenomena (neither in FC elements), this bijection not preserving length neither size should have even meaning in the other Catalan interpretations, highly likely a strong meaning.

\subsection{Sub-partitions (one parameter partitions of $C_{n+1}$):  Catalan triangle}

We consider the canonical form of a fully commutative element $w$. We wish to count the number of FC elements starting with $i:=i_1$ (a convenient shortcut for:  the canonical word for $w$ starts with $\sigma_i$ on the left). Let us call the set of such elements $^{i}A^c_n$ and call its cardinality $^{i}S_n$, in the same way we define $S_n^j$ where $j:=j_p$. Notice that $^{n}S_n = C_n $ and that $^{1}S_n = n$. We let 
  $^{0}A^c_n= \{ 1 \}$.
We have the following set-partition: 

 \begin{proposition} We have  $A^c_n =  \{1 \}\cup( \bigcup_{i=1}^{i=n}(^{i}A^c_n))$, and in details, for $1\le i \le n$: 
 
   \begin{equation}\label{part2}
    ^{i}A^c_n= \    ^{i}A^c_{n-1} \cup (\bigcup_{k=0}^{k=i-1} (\sigma_{i}\sigma_{i+1} \dots \sigma_{n} \  . \ ^{k}A^c_{n-1}). 
     \end{equation}

      \end{proposition}
     
 \begin{demo}
	Clear.
	\end{demo}

By letting  $^{i}S_n$ as above and considering the set partition (\ref{part2}) with the obvious initial conditions, we get: 
\begin{equation}\label{part2cardinalities}
 ^{i}S_n = \   ^{i}S_{n-1} + \sum_{k=0}^{k=i-1}   (^{k}S_{n-1}) =  \sum_{k=0}^{k=i}   (^{k}S_{n-1}). \ \text{Hence,}
\end{equation}

 $$
 ^{i}S_n  =  \     ^{i-1}S_{n} + ^{i}S_{n-1} .
 $$
 
 Basically this is the recurrence defining the Catalan triangle. The interesting fact here is that proving the following identity (and all the following partition identities for what it matters) is nothing but  verifying easily the proposed hypothesis resulting from giving small values for the parameters (here $i=1,2,3$) against a fairly simple recurrence relation as the last one, without any need for generating functions!
 
 \begin{corollary}\label{triangle} We have the following partitions (the second follows  under the action of  $\Delta$):
 
 $$
^{i}S_n  =   \frac{n+1-i}{n+1} \binom{n+i}{i}.
$$
 $$ S^{j}_{n}  = ^{n-j+1}S_{n} =  \frac{j}{n+1} \binom{2n-j+1}{n}.$$

\end{corollary}

 Moreover, we could make use of the set-partition (\ref{part1}) to see (via thick and thin elements this time) that $^{i}S_n $ verifies the following recurrence relation:
\begin{equation}
^{i}S_n =   \    ^{i}S_{n-1} + \sum_{k=1}^{k=i}    (^{i-k} S _{n-(k+1)} C_{k}), \text{ which is}  \ ^{i}S_n = C_i + \sum_{k=0}^{k=i-1}    (^{i-k} S _{n-(k+1)} C_{k}). 
\end{equation}

For instance, when we take the sum $ \sum_{i=0}^{i=n} $ on the two sides we arrive in two movements to the recurrence definition of Catalan numbers. On the other hand this partition is sufficient to prove Corollary \ref{triangle} with a bit more sophisticated calculations, yet it shows how the Catalan triangle is given as a function in Catalan numbers, while the Catalan number itself is given basically as a partition function in the  triangle as follows, using the last Proposition :

$$
C_{n+1} =   \sum_{i=0}^{i=n}  \   ^{i}S_n  \   =  \sum_{i=0}^{i=n}  
\   \frac{n+1-i}{n+1} \    \binom{n+i}{i} .
$$
\begin{remark}
 Partition (\ref{part2}) was the starting point of the work with C. Blondel \cite{Sadek_David},  investigating the Poincar\'e polynomial of $W(A^c _n)$,  in which the $q$-version of this partition has given some families  of polynomials which are far from being understood.\\

\end{remark}
%%%%%%%%%%%%

\subsection{Double partitions (two parameters partitions of $C_{n+1}$) and more}

Here we wish to treat sub-sub-partitions, that come from the ramifications of combining  partitions (\ref{part1}) and (\ref{part2}), as the self-explaining: $^{i, j_1}S_n$ (thus $S_n^{i_p, j}$) and $^{i} \overset{p}{S_{n}}$ (thus $\overset{p}{S_{n}  ^{j} }$). In addition to the  enumeration $ ^{i}S_n^{j}$ (thus $ ^{j_1}S_n^{i_p}$), this last one is to be treated aside. We point out that to avoid repeated calculations we mention only the results and the starting partitions when needed. Now take for instance $^{i_1, j_1}S_n$ considering  partition (\ref{part2}):
$$
^{i_1, j_1}S_n =  \      ^{i_1, j_1}S_{n-1} +  \sum_{k=0}^{k=i_1 -1}    \    ^{k} S _{n-1} ,
$$
which implies $ \ ^{i_1, j_1}S_n - \  ^{i_1-1, j_1}S_n = \ ^{i_1-1} S_{j_1-1}$. 
Notice that $^{i_1, i_1}S_n = C_{ i_1} $ and  $^{1, j_1}S_n = 1 $. We get the following

 \begin{corollary}
 
\begin{equation}
^{i_1, j_1}S_n = \frac{j_1 - i_1 +2}{j_1+1} \  \binom{j_1+i_1-1}{j_1}, \text{ hence}
\end{equation}

\begin{equation}
S_n^{i_p, j_p} =  \frac{j_p - i_p +2}{n  - i_p + 2} \   \binom{2n- j_p - i_p+1 }{n-j_p}.
\end{equation}\*

  \end{corollary}
  
  \begin{remark}
  
  We notice that $^{i_1, j_1}S_n $ is independent from $n$,  since any  $j_1 \leq n$ is playing the role of $n$, this number counts the Dyck paths in a rectangle of dimensions $(i_1 +1) \times  (j_1 +1)$ (lattice paths under $ x=y$), this result can be found in (\cite{Sta15} -A2.(a)). The special case   $i_1 = j_1$ counts  the usual  
Dyck paths   from $(0,0)$ to $(i_1+1, i_1+1)$, indeed the formula gives the Catalan number 
$C_i$ since $2=2$...   \end{remark}

We count now $^{i} \overset{p}{S_{n}}$ (with $p \leq i$ forcibly, i.e. $^{i} \overset{p}{S_{n}}=0$ for $p>i$). Starting with the  set-partition (\ref{part2}) we get 

$$
^{i} \overset{p}{S_{n}} =  \     ^{i} \overset{p}{S_{n-1}} +  \sum_{k=0}^{k=i -1}  \    ^{k} \overset{p-1}{S_{n-1}}. 
$$

 \begin{corollary}
 \begin{equation}
^{i} \overset{p}{S_{n}} =   \frac{n+1-i}{n+1-p} \  \binom{i-1}{p-1} \ \binom{n}{p}, \text{ so}
\end{equation}
\begin{equation}
\overset{p}{S_{n}  ^{j}}  =   \frac{j}{n+1-p} \  \binom{n-j}{p-1}  \   \binom{n}{p}. 
\end{equation}

  \end{corollary}
  
\begin{remark}
Essentially this two parameters number is a combination of the Narayana and Catalan triangle, clearly this number counts  FC elements of size $p$ starting with $i$, we can find it otherwise in \cite{Deutsch_1999} as the number of Dyck paths with $p$ peaks among which the $y$'s of the highest is $i$. We have $ \sum_{i=1}^{i=n}  (^{i} \overset{p}{S_{n}} ) =\sum_{i=p}^{i=n}  (^{i} \overset{p}{S_{n}} ) =  \C_n^p$ and  $\sum_{p=1}^{p=n}  (^{i} \overset{p}{S_{n}} )= \sum_{p=1}^{p=i}  (^{i} \overset{p}{S_{n}} )= \   ^{i}S_n$. And obviously  $\sum_{i=1}^{i=n}  \sum_{p=1}^{p=n}  (^{i} \overset{p}{S_{n}} )= C_{n+1}$. \\ 

   \end{remark}

Now consider $ ^{i}S_n^{j}$, this is the number of FC elements starting with $i$ and ending with $j$. The essential remark is the following: if $i\le j$, elements with $i_1=i$, $j_p=j$ and $p>1$ are thick (because 
for $1\le t \le p$, we have $j_t \ge j_p \ge i_1 \ge i_t$, and if $p>1$ at least  one inequality is strict). So we can apply the procedure in 
Lemma \ref{bijectionthick} and get $ ^{i}S_n^{j}= \  ^{i}S_{n-1}^{j-1}$ if $i<j$, 
and $ ^{i}S_n^{i}= \   ^{i}S_{n-1}^{i-1} +1$, so $$
 ^{i}S_n^{j} = \     ^{i}S_{n-(j-i+1)}^{i-1}  + 1   \text{ whenever   }  i \leq j. 
$$
If $i>j$ we imitate (\ref{part1}) and use again Lemma  \ref{bijectionthick} to get 
$$
 ^{i}S_n^{j} = \    ^{i}S_{n-1}^{j-1}  +  \sum_{k=j}^{k=i}  (^{i-k}S_{n-1-k} S_{k-1}^{j}) \text{ whenever   }  i>j.
$$
With the clear identities $^{1}S_n^{j} = 1 $ and $^{i}S_n^{1} = \   ^{i-1}S_{n-1}$. At this point it would not be a surprise that the two recurrences are to be unified, that is, 
writing  
$$
 ^{i}S_n^{j} = \    ^{i}S_{n-1}^{j} + \ \sum_{k=0}^{i-1}  \   ^{k}S_{n-1}^{j}  
$$
(either there is no $\sigma_n$, or the first block is $[i,n]$) and the analogous relation for 
$^{i-1}S_n^{j}$, we obtain: 
$$
 ^{i}S_n^{j} = \    ^{i-1}S_{n}^{j-1} + \    ^{i}S_{n-1}^{j}  .
$$
This relation plus the initial conditions determines $ ^{i}S_n^{j}$, so we can again guess and check, be it for a stable subset of $i$ and $j$. 

\begin{corollary}  We have $
 ^{i}S_n^{i-1} =  \binom{n}{i-1}-1
$ and, for \underline{$ i \leq j$}: 

$$
 ^{i}S_n^{j} =  \binom{n-j+i-1}{i-1}.
$$ 

\end{corollary}

\begin{remark}

This two parameters enumeration is a special one! It is not yet a partition since a closed formula for  $ ^{i}S_n^{j} $ could not emerge  when $j<i-1$! Yet it is to be highly considered (as for the other two parameters enumerations) because the second part of this work is dealing with the type $B_{n+1}$ in which every FC element is a multiple from the right of an $A^c_n$ element by a pure $B_{n+1}$ type element, so left and right parameters are crucial in the $B_{n+1}$ enumeration, hence in the blob algebra.

   \end{remark}

We would not stress more the various three, four or five parameters partitions, for example $ ^{i_{1}, j_1}  \overset{p}{S_{n}  }, \     \overset{p}{S_{n}  ^{i_p, j_{p}} }  $, we just point out that in the same way and using the  set-partition  (\ref{part2}) it is   easy to get them (guessing patterns as well as proving by induction).  \\

%%%%%%%%%%%%%%%%

\section{"The" explicit bijection between non-crossing diagrams and FC elements} 

We produce below the only bijection between fully commutative elements in $W(A_n)$ and 
non-crossing diagrams on $n+1$ strings that is compatible with the product of monomials in the Temperley-Lieb algebra on the one hand and the product of diagrams by concatenation  on the other hand. We denote it by $\mathcal D$.  

The existence and uniqueness of $\mathcal D$ are recalled in subsection \ref{defdiag}
  below, where   we give the main definitions and fix some vocabulary. In  subsection \ref{components}, we state the essential Proposition describing the main properties of $\mathcal D(w)$, for $w \in W^c(A_n)$. 
In subsection \ref{componentsDw} we prove this Proposition with a series of Lemmas. 
In subsection \ref{thebijection} we state the Theorem describing the bijection $\mathcal D$ and detail the two algorithms corresponding to $\mathcal D$ and $\mathcal D^{-1}$. In subsection \ref{pending} we discuss some consequences and   pending questions.

\subsection{The definitions}\label{defdiag} 	
We recall briefly the definitions and results, by  now classical, about the Temperley-Lieb algebra viewed as a diagram algebra, as written in   	\cite{West} 	
and \cite{Graham_Lehrer_1996} to which we refer   for proofs and details.  
We let $R$ be a commutative ring with $1$ and we let $\delta$ be a fixed element in $R$. 
 
\subsubsection{The monomial basis}

The set  $A_{n}^c$ indexes a basis, called {\it monomial basis}, of the Temperley-Lieb algebra $TL(A_n)=TL_{R,\delta}(A_n)$ of type $A_n$ and parameter $\delta$ over $R$. Actually this algebra   has a presentation by generators $e_1$, ... , $e_n$ and  relations 
$$
e_i^2= \delta e_i, \quad   e_i e_{i\pm 1} e_i= e_i, \quad e_ie_j=e_je_i \text{ if } |i-j|>1, 
\quad  
\text{ for } 1 \le i,j, \le n.
$$
For $w  \in  A_{n}^c$ written in its canonical form 
(\ref{eq:Stembridge}): $w=
  [i_1, j_1]  [ i_2, j_2 ]  \dots   [i_p, j_p]$, 
we denote by $e_{w}$ the monomial 
$e_{i_1} \cdots e_{j_1}   \dots e_{i_p} \cdots e_{j_p}$; in particular our notation identifies $e_{\sigma_i}$ with $e_i$. 
Those monomials constitute the monomial $R$-basis of $TL(A_n)$, and as always $e_{w}$ doesn't depend 
on the reduced expression.

\subsubsection{Non-crossing diagrams}     

A {\it non-crossing diagram} $D$   on $n+1$ strings consists of two rows of $n+1$  dots
in which 
\begin{itemize}
\item each dot is joined to just one other dot and 
\item none of the joins intersect
when drawn in the rectangle defined by the two rows of $n+1$ dots.
\end{itemize}

We draw the two rows on horizontal lines and we let 
$A=\{ 1, 2, \cdots, n+1\}$ be the set of dots on the "top line", numbered from left to right, and   $B=\{ 1', 2', \cdots, (n+1)'\}$  be the set of dots on the "bottom line", numbered from left to right.    
We use the bijection from $A$ to $B$ given by $x \mapsto x'$, and we consider the total order on 
$A \cup B$ given by 
$$
1 < 2 < \cdots < n+1 < 1' < 2' < \cdots < (n+1)'. 
$$ 
With this order we can consider each join as a pair $(x,y)$ of dots 
such that $x<y$,  and consider $D$ as a set of such pairs, that we call {\it arrows}.  
 By convention the two rows of dots are drawn one below the other so that we call an arrow $(x,y)$ with $x \in A$ and  $y \in B$ {\it vertical} if 
$y=x'$, {\it positive} if $y > x'$ and {\it negative} if $y < x'$.

It is convenient to use the vocabulary of oriented arrows: an element $(x,y)$ of $D$ represents an  arrow $x \rightarrow y$  directed from $x$ to $y$, the vertices $x$ and $y$ are called the endpoints of the arrow, $x$ is the {\it tail} of the arrow and $y$ the 
{\it head} of the arrow.   We often shorten {\it non-crossing diagram} to  {\it diagram}.

\subsubsection{Multiplication of non-crossing diagrams}

We can concatenate two diagrams from top to bottom. We reproduce here the definition and the very example of \cite{Graham_Lehrer_1996}, with $5$ strings:  
\begin{figure}[h]\label{concatenation}
			\centering
			
\begin{tikzpicture}[scale=1]
\input{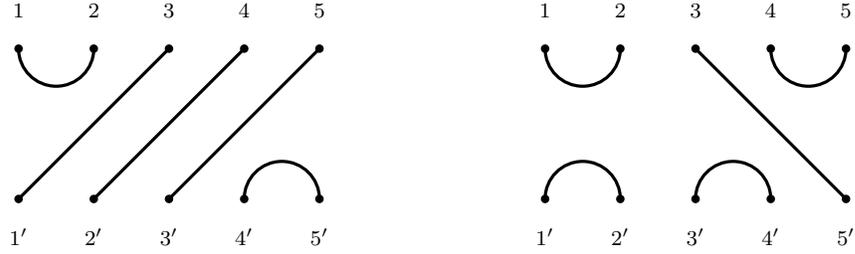}

\end{tikzpicture}

			\vspace{0.0cm}
			\caption{ $D_1$ (left),  $ \  D_2$ (right)} 
		 \end{figure}
%%\hfill 
\begin{figure}[h]   %%%\label{concatenation}
			\centering
			
\begin{tikzpicture}[scale=1]
\input{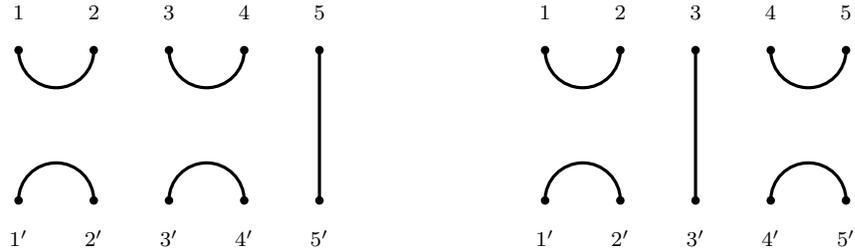}

\end{tikzpicture}

			\vspace{0.0cm}
			\caption{$D_1 \ast D_2$ (left), $\   D_2 \ast D_1$ (right)} 
		\end{figure}
if $D_1$ and $D_2$ are two diagrams, 
their concatenation  $D_1 \ast D_2$ is the diagram resulting from placing $D_2$ below $D_1$  and deleting the circles produced if any;  
then their product is 
$D_1 D_2 = \delta^{m(D_1,D_2)} D_1 \ast D_2$ where $m(D_1,D_2)$ is the number of deleted circles from the concatenation.

With this multiplication, the free $R$-module with basis the set of non-crossing diagrams on $n+1$ strings is given the structure of an $R$-algebra $\mathcal D_{R, \delta}(n+1)$, and the map sending $e_i$, 
$1 \le i \le n$,  to the diagram $$E_i=\{(i, i+1), (i', (i+1)')\}\cup \{(j, j')/ 1\le j \le i-1 \text{ or } i+2\le j \le n+1 \}$$ 
drawn below 

		\begin{figure}[h]
			\centering
			
\begin{tikzpicture}[scale=1]
\input{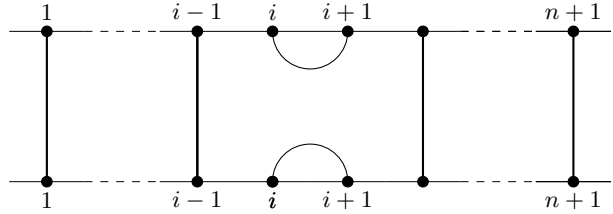}

\end{tikzpicture}

			\vspace{-0.0cm}
			\caption{Diagram $E_i$\label{DiagramFi}}
		\end{figure}
\noindent   
determines a unique isomorphism of $R$-algebras from $TL_{R, \delta}(A_n)$ to 
$\mathcal D_{R, \delta}(n+1)$. 
 This isomorphism   sends basis elements to  basis elements, hence determines the unique bijection
$\mathcal D$ between $A_n^c$, viewed as indexing the monomial basis,  and the non-crossing diagrams on $n+1$ strings, that  respects the product. We want to propose a direct, combinatorial description of this bijection $\mathcal D$. 

\subsubsection{Aside: another bijection between FC elements and  diagrams.}\label{cex} We propose now what should be $DF$ considering the terminology we adopted at the beginning. We explain why this bijection (and its dual) could not be "the" wanted one. First of all any $w$ in its canonical form is to be send via $FB$ to the ballot 
$$ \underbrace{+ \dots +} _{\text{($j_p$ copies)}} \underbrace{- \dots -} _{\text{($i_p$ copies)}}  \dots  \underbrace{+ \dots +} _{\text{($j_{1} - j_2 $ copies)}} \underbrace{- \dots -} _{\text{($i_{1} - i_2 $ copies)}} \underbrace{+ \dots +} _{\text{($n+1 - j_1 $ copies)}} \underbrace{- \dots -} _{\text{($n+1 - i_1 $ copies)}} ,$$ 
which can be seen clearly in Example \ref{ex5}, and since we are there we remark that  $FB ([4,5] [3,3] [1,1]) = +-++--++-+--,$ (its dual comes by reading the path from the top down to the bottom). 

Now every diagram can be sent to a ballot (see \cite{Sta15} ((59)) and ((61)) p.67) by sending every tail of every arrow to $+$ and every head to $-$ (its dual  comes from flipping the arrows). We see in Figure \ref{counterex} $NB (D_{[4,5] [3,3] [1,1]})$ and its dual.

\begin{figure}[h]
			\centering
			
\begin{tikzpicture}[scale=1]
\input{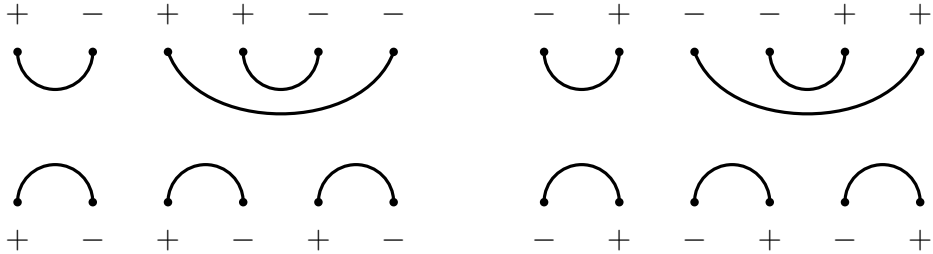}

\end{tikzpicture}

			\vspace{-0.5cm}
			\caption{Counterexample}\label{counterex}
		\end{figure}

Thus even if we take any dual of any of these bijections, we can never have  $ NF(D_{[4,5] [3,3] [1,1]})=[4,5] [3,3] [1,1]$, and our case is made.

\subsection{The main components of a diagram}\label{components}

Let $D$ be a diagram, we write $D$ as a disjoint union of four subsets: 
\begin{itemize}
\item $AD= \{(x,y) \in D \ / \  x,y \in A \}  \qquad$  (arrows in the top row), 
\item $BD= \{(x,y) \in D \ / \  x,y \in B \} \qquad$  (arrows in the bottom row), 
\item $D^+= \{(x,y) \in D \ / \  x\in A, y \in B,  x'<y \} \qquad$  (positive arrows from the top row to the bottom row), 
\item $D^-= \{(x,y) \in D \ / \  x\in A, y \in B, y\le x' \} \qquad$  (negative or vertical arrows from the top row to the bottom row).
\end{itemize}

We define: 
\begin{itemize}
\item $\mathcal T (AD)$ the set of tails of the arrows in $AD$, 

\item   
$\mathcal T (D^+)$ the set of tails of the arrows in $D^+$,   

\item    $\mathcal H (B D)$  the set of   heads  of the arrows in 
$BD$,  

\item  $\mathcal H (D^+)$  the set of   heads  of the arrows in  $D^+$;   

\item 
$ 
\mathcal I (D)= \{ i \in A \  / \  \exists y \in A \cup B \  (i,y) \in AD \cup D^+\}
= \mathcal T (AD)\cup  \mathcal T (D^+) $;

\item $\mathcal J (D)= \{ j  \in A \  / \  \exists x \in A \cup B \  (x, (j+1)') \in BD \cup D^+\}$ 

 $= \{ j  \in A \  / \ (j+1)' \in  \mathcal H (BD)\cup  \mathcal H (D^+) \} $ ; 
\item  
$P_D= \sharp (AD \cup D^+)$, called the {\it size} of $D$. 
 \end{itemize}

\begin{remarks}\label{product} {\rm 
\begin{enumerate}
\item $AD$ and $ BD$ have the same cardinality, hence 
$ 
\mathcal I (D)$ and $\mathcal J (D)$ also have the same cardinality. 
\item $D$ is fully determined by $AD$ and $ BD$, since there is a unique way to map elements of $A$ not in $AD$, to elements of $B$ not in $BD$, in a non-crossing way. 
(Indeed, after drawing $AD$ and $BD$,  we proceed from left to right and join at each step the leftmost free dot in $A$ to the leftmost free dot in $B$.) 
\item For arrows $(x,y)$ in $AD$ or $BD$, $x$ and $y$ have different parities. For arrows 
$(x,y)$ in $D^+ $ or $D^-$,  $x$ and $y$ have the same parity. 
\item 
It is straightforward to check that  $
AD \subset A(D\ast D') $ and $  BD'  \subset B(D\ast D')$.  \medskip
\end{enumerate}}
\end{remarks}

We need a definition before stating our main proposition. 
\begin{definition}
Let $(W,S)$ be a Coxeter system and let $w$ be an element of $W$. We say that 
$w$ is   {\it saturated} if 
$\text{ Supp} (w )= S$. 
\end{definition}

\begin{proposition}\label{describeDw}
Let $w \in W^c(A_n)$ be a fully commutative element, given in its canonical form (\ref{eq:Stembridge}): 
$$ 
w= [ i_1, j_1 ] [ i_2, j_2 ] \dots  [ i_p, j_p ]  , 
 \text{ with } 0\le p \le n  \text{ and } 
 \left\{ \begin{matrix}
 n\ge j_1 > \dots > j_p \ge 1 ,  \cr    n\ge i_1 > \dots > i_p \ge 1,   \cr
 j_t \ge i_t   \text{ for }  1\le t \le p.  
 \end{matrix}\right.$$
and let $\mathcal D(w)$ be the corresponding non-crossing diagram. Then 

\medskip

\begin{enumerate}
\item  $\mathcal I (\mathcal  D(w)) = \{ i_k / 1 \le k \le p \}$.  \medskip
 
\item  $\mathcal J (\mathcal  D(w) )= \{ j_k / 1 \le k \le p \}$.  \medskip

\item $\mathcal D(w)^+$ is the set of arrows $(i_s, (j_t +1)')$ where  $1 \le t < s \le p $ and 

\smallskip
\begin{itemize}

\item  $j_t = i_s + 2(s-t)-1$;

\item  for $t<r<s$ we have $j_r \ne i_s + 2(s-r)-1$
and $j_t \ne i_r + 2(r-t)-1$; \medskip

\item  $\Pi_{k=t+1} ^{k=s-1}   [ i_k, j_k ] $  is saturated in $ X_t ^{s} = <\sigma_{h}  \  /    i_{s} +1 \le h \le j_{t} -1>.$ \medskip
\end{itemize}

\smallskip
 
\item    We let $A(1)= \{i_1+1, i_1+2, \cdots, j_1+1  \}$, $f_1= \min A(1)=i_1+1$ and we define inductively on $r$, $1 < r \le p$,   the subsets:    \medskip

\begin{itemize}

\item if $ i_r \in \mathcal T (\mathcal D(w)^+)$:   $\   A(r)= \emptyset$;  \medskip

\item otherwise: 
$$ 
A(r)=   \{i_r+1, \cdots, j_1+1  \} - \{i_1, \cdots, i_{r-1}  \} 
- \{f_i  \ / 1\le i \le r-1, A(i) \ne \emptyset  \} .
$$

\end{itemize} 

If $A(r)$ is not empty, we let 
$f_r= \min A(r)$.   \medskip

Then $A\mathcal D(w)$ is the set of arrows $\{ (i_r, f_r) \ / \   1\le r \le p, 
A(r) \ne \emptyset  \} $.  \medskip 

\item We let $B(p)= \{ i_p,  i_p+1, \cdots, j_p \}$,  $g_p=\max B(p)=j_p$ and define inductively on $r$, $p> r \ge 1$,   the subsets:  
%%  \medskip

\begin{itemize}

\item if $ (j_r +1)' \in \mathcal H (\mathcal D(w)^+)$:   $\   B(r)= \emptyset$;  \medskip

\item otherwise: 
$$B(r)= \{ i_p,  i_p+1, \cdots, j_r \} - \{j_p+1, \cdots, j_{r+1}+1  \}- \{g_i  \ / p\ge i \ge r+1, 
B(i) \ne \emptyset  \}. $$

\end{itemize}

If $B(r)$ is not empty, we let 
$g_r= \max B(r)$.   
 
Then  $B \mathcal D(w)$  is the set of arrows $\{ (g'_r, (j_r+1)') \ / \   1\le r \le p, 
B(r) \ne \emptyset  \}$.  

\end{enumerate}
 
\end{proposition}

The following subsection is devoted to the proof of this Proposition. 

\subsection{The components of $\mathcal D(w)$}\label{componentsDw}

Let $w$ be as in Proposition \ref{describeDw}. We keep the notation in the Proposition. 

\subsubsection{Implementation of the induction on $p$}

Certainly   
$\mathcal D(1)$ is the identity diagram with arrows joining $x$ to $x'$ for  $1 \le x \le n+1$.   For $p=1$ we have $j_1\ge i_1$ and 
$w=\sigma_{i_1}  \dots \sigma_{j_1}$.  It is straightforward to see that the diagram below  is indeed the concatenation, from top to bottom, of the diagrams  
$E_{i_1}, E_{i_1+1}, \cdots,  E_{j_1}$ (Figure \ref{DiagramFi})  representing $\sigma_{i_1}, \sigma_{i_1+1}, \cdots, \sigma_{j_1}$.

		\begin{figure}[ht]
			\centering
			
\begin{tikzpicture}[scale=1]
\input{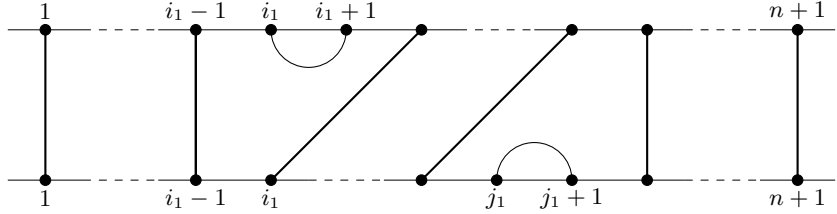}

\end{tikzpicture}

			\vspace{-0.0cm}
			\caption{$\mathcal  D( [ i_1, j_1 ]) $}\label{diagi1j1}
		\end{figure}  

It satisfies  $\mathcal I({\mathcal D( [ i_1, j_1 ])})= \{ i_1 \}$,  $ \mathcal J({\mathcal D( [ i_1, j_1 ])})= \{ j_1 \}$, $\mathcal D( [ i_1, j_1 ])^+= \emptyset$, as prescribed in  Proposition 
\ref{describeDw}.

\bigskip 
We let $\mathcal D= \mathcal D_p =\mathcal D(w)$ and, for $1\le s \le p$, 
we  let $w_s=  [ i_1, j_1 ]  \dots  [ i_s, j_s ]$ and 
$ \mathcal D_s =\mathcal D(w_s)$. For $1<s  \le p$ we have 
$w_{s}= w_{s-1} [ i_{s}, j_{s} ]$ hence 
 the diagram 
$\mathcal D_{s}$ is the concatenation  $\mathcal D_{s-1}\ast \mathcal D( [ i_{s}, j_{s} ])$ as illustrated below in Figure 
\ref{Concatenationisjs}. 

\medskip

\begin{figure}[h]
			\centering
			
\begin{tikzpicture}[scale=1]
\input{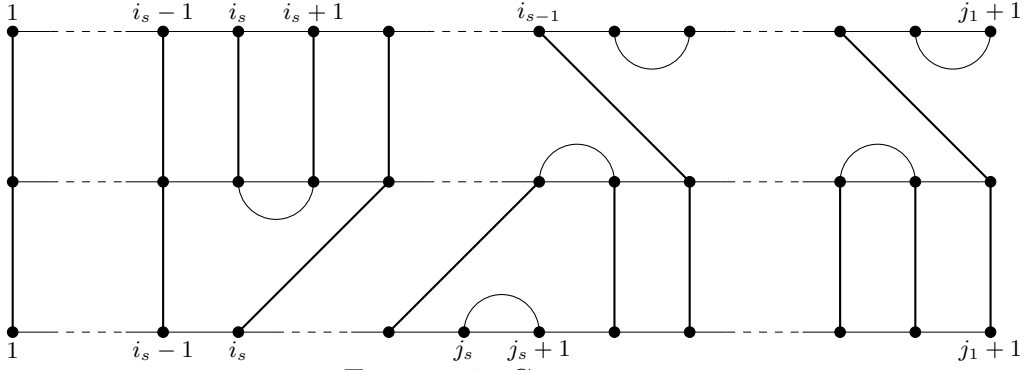}

\end{tikzpicture}

			\vspace{-0.5cm}
			\caption{Concatenation}
\label{Concatenationisjs}
		\end{figure}

\subsubsection{A series of Lemmas}

As remarked earlier in   Remark \ref{product} (4), we have: 
\begin{lemma}\label{LemmeAD}
$A\mathcal D_{s-1}  \subseteq  A\mathcal D_{s}$ and $(j_{s}', (j_{s}+1)')$ belongs to $B\mathcal D_{s}$.
\end{lemma}
Next we examine $\mathcal D_{s}^+$.
\begin{lemma}\label{LemmeDplus}
  $ \mathcal D_{s-1}^+  \subseteq   \mathcal D_{s}^+$.
\end{lemma}

\begin{proof}
Let $(x, y') \in \mathcal D_{s-1}^+$, so  $x,y \in A$, $x < y$. 
We note that   $ \mathcal D_{s-1}^-$ contains the arrows $(z, z')$  for $z < i_{s-1}$, so  $x$   satisfies $x \ge i_{s-1} > i_{s}$. Furthermore $y$ is a $j_t+1$ for some $t\le s-1$ so 
$y> j_{s}+1$, hence, when concatenating,  the arrow is followed downwards by a vertical arrow and  indeed we keep $(x, y') \in \mathcal D_{s}^+$.  
\end{proof}

\begin{lemma}\label{LemmeDmoins}
Let $(x, y') \in \mathcal D_{s-1}^-$, so  $x,y \in A$, $x \ge y$. Then 
\begin{enumerate}
\item If $y < i_{s}$ or $y > j_{s}+1$, we have $(x, y') \in \mathcal D_{s}^-$.
\item  If  $i_{s}+1 <y \le  j_{s}+1$, we have $(x, (y-2)') \in \mathcal D_{s}^-$.
\item  If $y=  i_{s}+1$ we obtain $(i_{s}, x) \in A\mathcal D_{s}$.  
\item If $y=  i_{s}$  we obtain,   
\begin{enumerate}
\item if $ i_{s} +1< i_{s-1}$:  $(i_{s}, i_{s}+1) \in   A\mathcal D_{s}$;
\item if $ i_{s}+1=i_{s-1} = j_{s-1}$:  $( i_{s}, ( i_{s}+2)') \in \mathcal D_{s}^+ $;
\item if $ i_{s}+1=i_{s-1} < j_{s-1}$: 
\begin{enumerate}
\item 
either  there is a $t<s$ such that $(i_{s-1}', (j_t+1)')$ 
 belongs to
$  B\mathcal D_{s-1}$, then 
$( i_{s}, (j_t+1)') $ belongs to $ \mathcal D_{s}^+ $;   

\item   or  there is an $a>i_{s-1}$ such that $(a, i_{s-1}')$  belongs to  $  \mathcal D_{s-1}^-$, then 
$( i_{s}, a) $ belongs to $  A\mathcal D_{s}$.
\end{enumerate}
\end{enumerate}
\end{enumerate}
\end{lemma}
\begin{proof} 
The first three statements are immediate from the figure describing the concatenation. We consider the fourth, where we must have $x=y= i_{s}$, we compose with the arrow $(i_{s}, i_{s}+1)$ in $\mathcal D( [ i_{s}, j_{s} ])$  followed by the arrow $\mathcal A$ in $\mathcal D_{s-1}$ containing 
$(i_{s}+1)'$. 
\begin{itemize}
\item  If $ i_{s}+1 < i_{s-1}$,  $\mathcal A = (i_{s}+1,( i_{s}+1)')$ producing 
$(i_{s}, i_{s}+1) $ in $  A\mathcal D_{s}$.

\item     If  $ i_{s}+1=i_{s-1} = j_{s-1}$, then $\mathcal A = (i_{s-1}',( i_{s-1}+1)')$ belongs to 
$  B\mathcal D_{s-1}$  and since $j_{s-1} > j_{s}$ this produces 
$( i_{s}, ( i_{s}+2)') $. 

\item     If  $ i_{s}+1=i_{s-1} < j_{s-1}$, then $i_{s-1}'$ is either the tail of  
$\mathcal A =(i_{s-1}', (j_t+1)')$ 
 in 
$  B\mathcal D_{s-1}$  or the head of   $\mathcal A =(a, i_{s-1}')$  in  $  \mathcal D_{s-1}^-$.
%% $\mathcal A $ cannot belong to 
%%$  B\mathcal D_{k}$ (for dots strictly lower than $i_k$ the arrows are vertical, 
%%so if $\mathcal A $  was to belong to $  B\mathcal D_{k}$ it would be some 
%%$(i_k', (j_u+1)')$ 
\end{itemize}
\end{proof}

\begin{remark}
 Case (4b) in the previous Lemma is a special case of (4c): indeed,   if $i_{s-1} = j_{s-1}$, the arrow  $(i_{s-1}', (j_{s-1}+1)')$ 
 belongs to
$  B\mathcal D_{s-1}$.
\end{remark}

\begin{lemma}\label{LemmeBD}
Let $(x', y') \in B\mathcal D_{s-1}$, so  $x,y \in A$, $x < y$. Then $y=j_t+1 > j_{s}+1$ for some $t\le s-1$ and $x \ge i_{s-1} >i_{s}$, and :   
\begin{enumerate}
\item if $x   > j_{s}+1$ then $(x', y') \in B\mathcal D_{s}$; 
\item if $i_{s}+1 < x   \le  j_{s}+1$ then  $((x-2)', y') \in B\mathcal D_{s}$; 
\item if $x=  i_{s}+1$, which implies $i_{s-1}=  i_{s}+1$,     then $(i_{s}, y')   \in \mathcal D_{s}^+$. 
\end{enumerate}
\end{lemma}
\begin{proof}
Clear. 
\end{proof}

\subsubsection{Proof of Proposition  \ref{describeDw}, assertions (1) and (2)}

So far we have proved the following: 
\begin{itemize}
\item  $A\mathcal D_{s-1}  \subseteq  A\mathcal D_{s}$  and 
 $ \mathcal D_{s-1}^+  \subseteq   \mathcal D_{s}^+$,   in particular  
$\mathcal I(\mathcal D_{s-1}) \subseteq \mathcal I(\mathcal D_{s}).$ 
\item  
All the arrows in $\mathcal D_{s-1}^-$ produce arrows in  
$\mathcal D_{s}^-$, except one, that produces  an arrow
$(i_{s}, x) $ in $ A\mathcal D_{s}$ in most cases, or else an arrow 
$( i_{s}, (j_t+1)') $ in $  \mathcal D_{s}^+ $ for some $t< s$, exactly in the case when $i_{s-1}= i_{s}+1$ and $(i_{s-1}', (j_t+1)')$ 
 belongs to
$  B\mathcal D_{s-1}$. 

\item The heads of arrows in $ B\mathcal D_{s-1}$  remain heads of arrows in $B\mathcal D_{s}$ except possibly 
if $i_{s-1}= i_{s}+1$ and $i_{s-1}'$ is the tail  of an arrow in $ B\mathcal D_{s-1}$, in which case we obtain the head of the arrow $(i_{s}, y')$ in 
$\mathcal D_{s}^+$. 
 In particular  
$\mathcal J(\mathcal D_{s-1}) \subseteq \mathcal J(\mathcal D_{s}).$

\item The arrows in  $B\mathcal D_{s}$  are either 
 $(j_{s}', (j_{s}+1)')$, or obtained from an arrow in $B\mathcal D_{s-1}$  by concatenation, which does not change the head of the arrow. 
\end{itemize} 

\medskip

From this we deduce $ \mathcal I(\mathcal D_{s})=  \mathcal I(\mathcal D_{s-1}) \cup \{  i_{s}\}$  hence by induction assertion (1) in Proposition \ref{describeDw}. The possible arrows with tail $i_{s}$ are detailed in Lemma \ref{LemmeDmoins}. 

\medskip 

Since the arrow $(j_{s}', (j_{s}+1)')$ belongs to $B\mathcal D_{s}$ by Lemma \ref{LemmeAD},  we obtain  $ \mathcal J(\mathcal D_{s})=  \mathcal J(\mathcal D_{s-1}) \cup \{  j_{s}\}$  hence by induction assertion (2) in Proposition \ref{describeDw}. 

\medskip 
\subsubsection{Proof of Proposition  \ref{describeDw}, assertion (3)}

We start on the proof of assertion (3). For this  we examine more closely the conditions necessary to obtain an arrow in $\mathcal D_{s}^+$ that does not belong to 
$\mathcal D_{s-1}^+$. Such an arrow, if any, is unique and has the form 
 $(i_{s}, (j_t+1)')$ for some $t<s$. 
The  last two Lemmas produce the same necessary and sufficient condition: 
\begin{equation}\label{conditionDplus}
 i_{s}+1=i_{s-1} \text{ and }  (i_{s-1}', (j_t+1)') \in    B\mathcal D_{s-1} 
\text{ for some } t<s .   
\end{equation}
We have to understand the formation of $ B\mathcal D_{s-1}$ to find the algebraic meaning of this condition.  We distinguish two cases. 

\begin{itemize}
\item{Case 1:}     $t=s-1$,  i.e.  $(i_{s-1}', (j_{s-1}+1)') \in    B\mathcal D_{s-1}$. This 
 condition
 is equivalent to $j_{s-1}=i_{s-1}$. 
(Indeed $(j_{s-1}', (j_{s-1}+1)') $ always belongs to $    B\mathcal D_{s-1}$.)\\

\item{Case 2:}    for some $t<s-1$,  $ (i_{s-1}', (j_t+1)') $ belongs to $    B\mathcal D_{s-1}$. Then 
 the arrow $ (i_{s-1}', (j_t+1)') $  must enclose the $s-1-t$ arrows in $ B\mathcal D_{s-1}$ with heads $(j_r+1)'$ for $t< r \le s-1$ so we must have $ j_t+1= i_{s-1}+ 2(s-1-t)+1 $,      or equivalently 
\begin{equation}\label{ju}
j_t= i_{s}+ 2(s-t) -1.   
\end{equation} 
%%=i_{k+1}+ 2(k+1-u)  $. \\  
Furthermore, the arrow $(i_{s-1}', (j_t+1)') $ must come  from $(j_t', (j_t+1)') \in  B\mathcal D_{t}$ through a series of $ s-1-t$   concatenations. According to Lemma \ref{LemmeBD} each concatenation may substract $2$ from the tail, starting with $j_t$. 
Comparing with  equality \eqref{ju}    we see that the substraction must occur each time, which means that  each iteration  falls 
into condition (2) in Lemma \ref{LemmeBD}, so that  for each $r$, 
$t\le r< s-1$, we have, after $r-t$ iterations: 
$$ 
i_{r+1}+1 < j_t - 2 (r-t)   \le  j_{r+1}+1
$$
so that the next iteration provides the arrow 
$(( j_t - 2 (r+1-t) )', (j_t+1)')$. Note that the double inequality above is   equivalent, with \eqref{ju}, to: 
\begin{equation}\label{minus2}
i_{r+1} < i_{s} + 2 (s-1-r)   \le  j_{r+1} \text{ for }  t\le r< s-1.
\end{equation}
This condition implies   $i_{r} <    j_{r}$  for each $r$, 
$t < r\le s-1$. Furthermore, if \eqref{minus2} is satisfied, then from 
the same iterations as above we arrive at $(i_{s-1}, (j_t+1)')
\in    B\mathcal D_{s-1}$, and adding 
$ i_{s}+1=i_{s-1}$ we obtain 
 $(i_{s}, (j_t+1)')
\in \mathcal D(w)^+$ (again from Lemma \ref{LemmeBD}). 
\end{itemize}

We subsume this  in the following assertion: 

\smallskip
{\it The pair 
 $(i_{s}, (j_t+1)')$,   $t <  s$, is an arrow in $\mathcal D(w)^+$ if and only if 
\begin{equation}\label{resume}
\begin{aligned}
&j_t= i_{s}+ 2(s-t) -1, \\
 &i_{s}+1=i_{s-1} ,  \\
& t<r\le s-1  \implies  i_{r} < i_{s} + 2 (s-r)   \le  j_{r} . 
\end{aligned} 
\end{equation} 

All the arrows in $\mathcal D(w)^+$ have this form. }

\smallskip 
  
To obtain assertion (3) in Proposition  \ref{describeDw}, we need to show that conditions \eqref{resume} above are equivalent to the conditions in assertion (3),  easier to check on any fully commutative element. 

We assume first that  conditions \eqref{resume} hold. If $t=s-1$, we do have the only condition in assertion (3), namely 
$j_{s-1}= i_{s}+  1$. If $t<s-1$ we have, for $t<r<s$,
$ i_{s} + 2 (s-r)   \le  j_{r}$ hence certainly $ i_{s} + 2 (s-r) -1  \ne  j_{r}$. If we rewrite the condition with $j_t$ instead of $i_s$ we find 
$i_r < j_t -2(r-t)+1$ hence certainly 
$j_t \ne i_r+2(r-t)-1$. 
We check the saturation, which amounts to: 
\begin{itemize}
\item $j_{t+1}=j_t-1$: indeed in the inequality 
$i_{s} + 2 (s-t-1)   \le  j_{t+1}$ the left term is 
$j_t-1$, whereas we always have $j_{t+1} < j_t$.   
\item for $t+1< r  \le s-1$,    $j_r \ge i_{r-1}-1$:    assume for a 
contradiction that there is such an $r$ with  $j_r < i_{r-1}-1$, then we get 
$ i_{s} + 2 (s-r) \le  j_r < i_{r-1}-1$. Since 
$r-1>t$  we have 
$i_{r-1}-1 < i_s + 2(s-r)+1  $  so we find a difference of  at least 2 between $ i_s + 2(s-r)+1$ and $ i_s + 2(s-r)$, absurd. 
\item  $i_{s-1}= i_{s}+  1$: assumed in \eqref{resume}. 
\end{itemize}
So indeed conditions \eqref{resume} imply the conditions in   assertion (3). 

\medskip

Now we take a pair $(i_{s}, j_t)$,   $t <  s$, satisfying the conditions in assertion (3) of the Proposition, namely: 
\begin{itemize}

\item  $j_t = i_s + 2(s-t)-1$,

\smallskip

\item{minimality:} for $t<r<s$ we have $j_r \ne i_s + 2(s-r)-1$     
 and $j_t \ne i_r + 2(r-t)-1$,

 \smallskip

\item{saturation:}  $\Pi_{k=t+1} ^{k=s-1}   [ i_k, j_k ] $  is saturated in $ X_t ^{s} = <\sigma_{h}  \  /    i_{s} +1 \le h \le j_{t} -1>$,   
\end{itemize}
and we prove that 
it satisfies conditions \eqref{resume}.

\smallskip 
The first equality $j_t= i_{s}+ 2(s-t) -1$ is an assumption. If $t=s-1$ it implies 
$j_{s-1}= i_s+1$ hence $i_{s-1}= i_s+1$ is mandatory and \eqref{resume} is proven.

If $t<s-1$, the equality  $i_{s-1}= i_s+1$ is implied by the 
saturation condition. 
It remains to   establish the inequalities 
$  i_{r} < i_{s} + 2 (s-r)   \le  j_{r}$  for 
$t<r\le s-1$. 
  
\smallskip

For $r=t+1$ they read   
$  i_{t+1} < i_{s} + 2 (s-t) -2 \le  j_{t+1}  $ 
whereas $ i_{s} + 2 (s-t) -2= j_t-1$ and $ j_t-1=j_{t+1}$ from the saturation condition. We just need to rule out  $  i_{t+1}=  j_{t+1}$. 
But if this was true, we would have $j_t=  i_{t+1}+	1$  hence, by minimality, $s=t+1$, which has been excluded. So \eqref{resume} holds. 

We prove the leftmost inequality   by ascending induction, assuming that it  holds from $t+1$ to $r-1$. 
We prove it for $r\le s-1$. We assume for a  contradiction that it   doesn't hold. 
Then   $  i_{r} \ge i_{s} + 2 (s-r) $. The inequality for $r-1$ gives 
$  i_{r-1} < i_{s} + 2 (s-r) +2 $, implying 
$i_r <  i_{s} + 2 (s-r) +1 $, so the only possibility is $  i_{r} = i_{s} + 2 (s-r) =    j_t - 2 (r-t) +1$, hence 
$j_t= i_r + 2(r-t) -1$, contradicting the minimality conditions. 

\smallskip 
For $r=s-1$ the rightmost inequality reads 
$i_s+2\le j_{s-1}$. The saturation condition implies 
$i_{s-1}=i_s+1$ hence $j_{s-1}\ge i_s+1$. Now 
$j_{s-1}= i_s+1$ contradicts the minimality conditions, so indeed 
$i_s+2\le j_{s-1}$ holds. We proceed by descending induction, assuming that the rightmost inequality holds from 
$s-1$ to $r+1$, and we prove it for $r$, $t  <r < s-1$.  

Assume for a   contradiction that it does not hold, then $  j_{r} < i_{s} + 2 (s-r) $. We compare this with the  inequality holding for  $r+1$: $     i_{s} + 2 (s-r) -2 \le j_{r+1}$ and get 
$     i_{s} + 2 (s-r) -2 \le j_{r+1} < j_r < i_{s} + 2 (s-r) $,  
which implies $ j_r = i_{s} + 2 (s-r) -1$, forbidden by the minimality conditions. 

Finally \eqref{resume} holds as we claimed.

\subsubsection{Proof of Proposition  \ref{describeDw}, assertions (4) and (5)}

Assertion (4) in Proposition  \ref{describeDw} is the description of $A\mathcal D(w)$. We prove it by induction on $s$ as before. 
Certainly $A\mathcal D_1 = \{(i_1, i_1+1)\}$ by 
Figure \ref{diagi1j1}. Let $s$ such that $1 < s\le p$, we assume that (4) holds for $w_{s-1}$ and we prove it for $w_s$. Thanks to Lemmas \ref{LemmeAD} and \ref{LemmeDplus} we only have to prove that: 

\smallskip
\centerline{\it 
 if $ i_s $ is not the tail of an arrow in  $ \mathcal D_s^+$  (or 
 $ \mathcal D(w)^+$),}

\centerline{\it  then 
$A(s)$ is not empty and $(i_s, f_s)$ belongs to $A\mathcal D_s$. }   

\smallskip   
The condition on $i_s$ is the negation of condition \eqref{conditionDplus}, so 
\begin{itemize}
\item either $i_s+1 <i_{s-1}$, in which case indeed 
$A(s)$ is not empty and $f_s=i_s+1$, and by Lemma \ref{LemmeDmoins}  (4)(a),   $(i_s, i_s+1)$ belongs to $A \mathcal D_s$.  

\item or $i_s+1 =  i_{s-1}$ and $i_{s-1}'$ is not the tail of an arrow in 
$B\mathcal D_{s-1}$, i.e. $i_{s-1}'$ is the head of an arrow 
in $\mathcal D_{s-1}^+$ or  $\mathcal D_{s-1}^-$. The first is impossible, it must be the second, so we have an arrow 
$(x, i_{s-1}')$ with $x>i_{s-1}$ (equality is impossible because $i_{s-1}$ is the tail of an arrow in $A\mathcal D_{s-1} \cup \mathcal D_{s-1}^+$).    So $x$ is neither an $i_k$ nor an $f_k$ if any for $k \le s-1$, and $x\ge i_s+1$, so $x $ belongs to $A(s)$ which is non empty. The concatenation produces the arrow 
$(i_s,x) \in A\mathcal D_s$ so we want to prove that $x = \min A(s)$. 
Suppose otherwise, then $i_s+1= i_{s-1}  < f_s <x$. 
Since $(x, i_{s-1}')$ is an arrow in $\mathcal D_{s-1}^-$, and any dot 
left of $ i_{s-1}'$ is the head of a vertical arrow in  $\mathcal D_{s-1}^-$, all dots between $i_{s-1} $ and $x-1$ must belong to arrows in 
$A\mathcal D_{s-1}$, which are all of the form $(i_k, f_k)$ with $k\le s-1$ by induction hypothesis. So $f_s$, that cannot be an $i_k$ with
$k\le s-1$, has to be an $f_k$ with $k\le s-1$, impossible by the definition of $A(s)$. 
\end{itemize} 

Assumption (4) is proven.  

\medskip  

We terminate with the proof of assertion (5), which is the description of  $B\mathcal D(w)$. With what we know at this point, (5) is equivalent to the following assertion: 

\smallskip 
\centerline{\it Let $1 \le r \le p$ and assume $(j_{r}+1)'$ 
is not the head of an arrow in $\mathcal D(w)^+$.}

\centerline{\it Then 
$B(r) $ is not empty and, letting $g_r=\max B(r)$, 
$(g_r', (j_{r}+1)')$ belongs to $B\mathcal D(w)$.}

\smallskip 

 We already know that $(j_{s}', (j_{s}+1)')$ belongs to $B\mathcal D_{s}$, 
so indeed $(j_{p}', (j_{p}+1)')$ belongs to $B\mathcal D(w)$, which is the case $r=p$. 
We now work by descending induction on $r$: let $r$ with 
$1\le r <p$, we assume our assertion holds for $r-1$ and prove it for $r$.   

By definition, $B\mathcal D(w)$
is the set of arrows in  $\mathcal D(w)$ with head some $(j_{r}+1)'$ and not belonging to $\mathcal D(w)^+$, so certainly if 
$(x', (j_{r}+1)')$ belongs to $B\mathcal D(w)$, then $x \le j_r$, $x$  is not a $j_k+1$, and is not a $g_k$ for $k \ge r+1$ either, hence $x $ belongs to $B(r)$ which is not empty. We must prove that  $x = \max B(r)$, so we assume   for a contradiction that $x < g_r$. Then the arrow $(x', (j_{r}+1)')$  strictly encloses an arrow in $B\mathcal D(w)$  with head or tail $g_r'$, actually with tail $g_r'$ since $g_r$ is not a 
$j_k+1$. The head of this arrow must be a $(j_k+1)'$  strictly less than $(j_{r}+1)'$, so $k > r$ and by induction we have $g_r=g_k$ which is impossible by construction of $B(r)$. 

This finishes the proof of assertion (5) and finally of Proposition  \ref{describeDw}.

\subsection{``The'' bijection and the corresponding algorithms}\label{thebijection}

 Proposition  \ref{describeDw} ascribes a non-crossing diagram 
on $n+1$ strings  to a fully commutative element in $W(A_n)$ in a unique way: all constitutive elements of the diagram $\mathcal D(w)$ are described directly in  terms of $w$, or rather, in terms of the canonical form of $w$. So we can describe the inverse of the bijection $w \mapsto \mathcal D(w)$, that we denote by 
$D \mapsto w_D$,   as follows: 

\begin{proposition}
Let $D$ be a non-crossing diagram on $n+1$ strings, of size $p$ and 
write $\mathcal I(D)$ and $\mathcal J(D)$ in decreasing order, that is: 
$$\begin{aligned}
\mathcal I(D)&= \{ i_1,  \cdots , i_p\} \quad \text{ with } i_1> i_2  > \cdots >i_p,  \\  
\mathcal J(D)&= \{ j_1, \cdots , j_p\} \quad 
\text{ with }    j_1> j_2  > \cdots >j_p. 
\end{aligned}$$
Then for all $k$, $1 \le k \le p$, we have $i_k \le j_k$ and 
$ D$ is the image under $(\mathcal D)^{-1}$ of the fully commutative element 
$$
w_D=  [ i_1, j_1 ] [ i_2, j_2 ] \dots  [ i_p, j_p ] .  
$$ 
\end{proposition}

There is nothing to prove here, it is just a consequence of Proposition \ref{describeDw} and the fact that we are dealing with two finite sets of same cardinality.  We see in particular that $\mathcal I(D)$ and $  
\mathcal J(D)$ determine $D$ entirely. 
We   write down in the next subsections the actual algorithms 
associated to the bijection $\mathcal D$ and its inverse. Meanwhile we conclude  with a convenient statement.

\begin{theorem}\label{FC D}
 Let  $w$ be 
a FC element  of $W(A_n)$ given in its normal form (\ref{eq:Stembridge}): 
$$ 
w= [ i_1, j_1 ] [ i_2, j_2 ] \dots  [ i_p, j_p ]  , 
 \text{ with } 0\le p \le n  \text{ and } 
 \left\{ \begin{matrix}
 n\ge j_1 > \dots > j_p \ge 1 ,  \cr    n\ge i_1 > \dots > i_p \ge 1,   \cr
 j_t \ge i_t   \text{ for }  1\le t \le p.  
 \end{matrix}\right.$$

\begin{enumerate}
\item There is a unique non-crossing diagram $\mathcal D(w)$ with  $\mathcal I (\mathcal  D(w))  = \{ i_1, \cdots, i_p\}$ and  $ \mathcal J (\mathcal  D(w))  = \{ j_1, \cdots, j_p\}$. 

\medskip
\item  The map $w \mapsto \mathcal D(w)$   is the unique bijection between the set of FC elements of $W(A_n)$ and the set of non-crossing diagrams of size $2(n+1)$   that is compatible with the  monomial multiplication and 
the multiplication of non-crossing diagrams. 
\end{enumerate}
\end{theorem}

\subsubsection{First algorithm: From diagrams to FC elements} 
It is very easy to write the FC element that corresponds to a given diagram $D$, since it is given by $\mathcal I(D)$ and $\mathcal J(D)$: 
\begin{itemize}
\item Single out  on the top line the tails of arrows going strictly right and list their numbering in decreasing order: $\{i_1, \cdots, i_p\}$. 
\item Single out on the bottom line the heads of arrows coming strictly from the left and list their numbering {\bf minus $\mathbf 1$} in decreasing order: $\{j_1, \cdots, j_p\}$. 
\item The FC element that corresponds to $D$ is 
$[ i_1, j_1 ] [ i_2, j_2 ] \dots  [ i_p, j_p ] $.
\end{itemize}

\subsubsection{Second algorithm: From FC elements to diagrams}
Again we subsume the previous subsections, starting with a FC element $w$. 
\begin{itemize}
\item Write $w$ in its canonical form 
$w=[ i_1, j_1 ] [ i_2, j_2 ] \dots  [ i_p, j_p ] $ and mark the dots 
$i_k$ (upper line) and $j_k + 1 $ (lower line). 

\item For any $u$ such that $1\le u < i_p$ or $j_1+1 <u \le n+1$, draw the arrow 
$(u, u')$. 

\item Working from right to left, i.e. for increasing $s$, do the following for each $s$, 
$1\le s \le p$: 

\smallskip 
\begin{itemize}
\item Check    the following conditions:  

\smallskip

\begin{enumerate}
\item  $i_s+1=i_{s-1}$, 

\item  there exists a highest index $t\le s-1$ 
such that  $j_{t}=i_{s-1}+2(s-1-t)$ and  $(j_t+1)'$ is a free dot
 (i.e. not already the head of an arrow),   

\item  the element 
$[ i_{t+1}, j_{t+1} ] \dots  [ i_{s-1}, j_{s-1} ] $ is saturated in the group generated by the $\sigma_k$ with 
$i_{s-1} \le k \le j_t-1$. 

\end{enumerate}

\smallskip\noindent
\item If those three conditions are satisfied, 
 draw the arrow 
$(i_s, (j_t+1)') $.  

\smallskip
\item If $s<p$ proceed to $s+1$. 

\end{itemize}

\smallskip 

\item Draw the arrow $(i_1, i_1+1)$ in $A\mathcal D(w)$, then, proceeding from right to left, i.e. on increasing order on $k$, $1\le k \le p$,  for each
 $k $ such that $i_k \notin \mathcal T(\mathcal D(w)^+)$ (i.e. for each $k$ such that $i_k$ is not already the tail of an arrow),  
 find $f_k$ the lowest free dot in $A$ on the right of $i_k$ 
{\it which is not an $i_s$},  and draw  $(i_k, f_k)$ in $A\mathcal D(w)$.

\item Draw the arrow $(j_p', (j_p+1)')$ in $B\mathcal D(w)$, then, proceeding from left to right, i.e. on decreasing order on $k$, $1\le k \le p$,  for each $k$  such that $(j_k+1)' \notin \mathcal H(\mathcal D(w)^+)$   (i.e. for each $k$ such that $(j_k+1)' $ is not already the head of an arrow),
find $g_k'$ the  highest free dot in $B$ on the left of $(j_k+1)'$ 
{\it which is not a $(j_s+1)'$}  and  draw  $(g_k', (j_k+1)')$   in $B\mathcal D(w)^+$.

\item Proceed from left to right to draw the remaining arrows, all going from  top  to  bottom, and forcibly all vertical or negative. Each arrow joins the lowest free dot on the top line to the lowest free dot on the bottom line. \\
\end{itemize}
		
		\subsection{Applications and questions}\label{pending}

\subsubsection{Monomial multiplication}
Multiplication in the monomial basis may be quite intricate to deal with. 
We have an alternative here: we can use the algorithm to transform two monomials (of course after putting them in their canonical forms, which is fairly easy) into diagrams, then we multiply the diagrams, an easy step, then we use the algorithm backwards to get the resulting monomial. We invite the careful reader to obtain immediately the following products
of  monomials:  $e_{[1,4]}. e_{[4,4][3,3][1,1]}= \delta e_{[3,3][1,1]}$ and 
 $ e_{[4,4][3,3][1,1]}.e_{[1,4]}= \delta e_{[4,4][1,1]}$; and
 to imagine the efficiency of this method for 20 or 30 generators! Moreover  the resulting monomial  comes in its canonical form as an output.

\subsubsection{Left and right multiplication}
It is easy to identify $\mathscr{L} (w)$ and  $\mathscr{R} (w)$ from the diagram $\mathcal D(w)$ attached to a FC element  $w$. We only need the following Lemma: 
\begin{lemma}   Let  $ 
w= [ i_1, j_1 ] [ i_2, j_2 ] \dots  [ i_p, j_p ]$ be 
an element  of $A_n^c$ given in its normal form (\ref{eq:Stembridge}). 
Then 
$$\begin{aligned}
\mathscr{L} (w) &=  \{ \sigma_{i_1}\} \cup \{  \sigma_{i_k}/ \    2 \le k \le p \text{ and } 
i_k < i_{k-1}-1  \} \\
\mathscr{R} (w) &=  \{ \sigma_{j_p}\} \cup \{  \sigma_{j_k}/ \    1 \le k \le p-1 \text{ and } 
j_k >j_{k+1}+1  \}
\end{aligned}
$$ 
\end{lemma}
\begin{proof} We sketch the proof for $\mathscr{L} (w)$, the other case is similar. 
The set $\mathscr{L} (w)$ is included in $\{  \sigma_{i_k}/   1 \le k \le p    \}$, as we prove now by induction on $p$. If $p=1$ the element is rigid, left-reduced only by $\sigma_{i_1}$. For $p\ge 2$, assuming the assumption holds up to $p-1$, we multiply $w$ on the left by $\sigma_s$ for $1 \le s \le n$ and assume 
$$ \text{(H) } \quad  \ell (\sigma_s w)<\ell(w).$$   

 If $\sigma_s$ commutes with $[i_1,j_1]$, that is if $i_1-1>s $ or $ s > j_1+1$, then $\ell (\sigma_s [i_1,j_1])  >\ell([i_1,j_1])$ so if we apply the exchange condition \cite[IV.1.5]{Bourbaki_1981}  to  $
\sigma_s  [ i_1, j_1 ] \dots  [ i_p, j_p ]$, the index of the element of $[ i_1, j_1 ] \dots  [ i_p, j_p ]$  giving rise to the exchange condition 
is at least $j_1-i_1+2$, so that we can cancel out $[i_1,j_1]$ on both sides in the exchange condition and find that 
$\ell (\sigma_s  [ i_2, j_2 ] \dots  [ i_p, j_p ])< 
\ell( [ i_2, j_2 ] \dots  [ i_p, j_p ])$, 
then
 the induction hypothesis applies and we find that $s=i_k$ for some 
$k>1$. 

Otherwise we have  $i_1-1\le s \le j_1+1$. If $s=j_1+1$ we obtain a new canonical form of size $p+1$ and first block $\sigma_{j_1+1}$ so we can't have (H). If $s=i_1-1$  we obtain a canonical reduced form of size $p$ for an element that may not be FC -- we recall that  any element in $W(A_n)$ has a canonical form similar to  (\ref{eq:Stembridge}), except that the condition 
$i_1>i_2 > \cdots > i_p$ is removed -- so we can't have (H). If $i_1+1\le s \le j_1$ we have   
$\sigma_s [i_1,j_1]=[i_1,j_1] \sigma_{s-1}$, reduced,  and based on the exchange condition as before, if $\sigma_s$ belongs to $\mathscr L(w)$, then  $\sigma_{s-1}$ must belong to 
$\mathscr L ([i_2,j_2] \cdots [i_p,j_p])$ hence $s-1$ is equal to  $i_k$ for some $k\ge 2$, impossible since  $s-1\ge i_1$. The only possibility left is   $s=i_1$. 

Now the set in the Lemma   is certainly  included in $\mathscr{L} (w)$, so it remains  to prove that if $i_k = i_{k-1}-1 $, then $\ell(\sigma_{i_k}w)>\ell(w)$. But certainly 
$\sigma_{i_k}$ commutes with the blocks $1$ to $k-2$, so we can write 
$$\sigma_{i_k}w= [i_1,j_1] \cdots [i_{k-2},j_{k-2}] [ i_{k-1}-1,j_{k-1}] [ i_{k-1}-1,j_{k}]  \cdots [i_p,j_p] $$
which is a canonical reduced form for a non-FC element, hence reduced. 
\end{proof}

From the algorithm linking an FC element to a diagram, we see that 
$(i, i+1) $ belongs to $ AD$ if and only if $i=i_1$ or $i$ is an $i_k$ and $i+1 $ is not, and similarly for $BD$, so:  

\begin{proposition}\label{LwRw} 
In the diagram $\mathcal D(w)$ attached to a FC element  $w$, 
one can read $\mathscr{L} (w)$ and  $\mathscr{R} (w)$  as follows: 
$$
\mathscr{L} (w) = \{ \sigma_i \ / \  (i, i+1) \in A\mathcal D(w) \}, \qquad \mathscr{R} (w)=\{ \sigma_j \ / \  (j', (j+1)') \in B\mathcal D(w) \}. $$ 
\end{proposition}
In other words, the elements of  $\mathscr{L} (w)$ (resp. $\mathscr{R} (w)$) correspond to  the tails of the shortest top  (resp.   bottom) bubbles. 
This Proposition  is to be considered in a graphic way and in relation to the monomial basis.

\subsubsection{Equivalence of diagrams}
We introduce the following equivalence on diagrams, hence on $A^{c}_n$ via the  bijection  of Theorem \ref{FC D}:   we say that two diagrams $D_1$ and $D_2$ are equivalent, and we write 
$D_1 \approx D_2$,  if   $$D^+ _1 \cup D^- _1 = D^+ _2 \cup D^- _2 \quad \text{ that is, }
D^+ _1  = D^+ _2  \text{ and }   D^- _1 =   D^- _2 .$$
This clear equivalence gives a mysterious equivalence on $A^{c}_n$,  and actually on every  $A^{c,p}_n$, since  obviously the two corresponding FC elements $w_{D_1}$ and $w_{D_2}$ have the same size. 
Note that we have given above an algebraic description of the elements of $D^+$.

\subsubsection{Questions}
The two algorithms raise a lot of questions, mainly about the interpretation on one side of properties known on the other side. As of now, we let the reader think about the following questions and remarks. 

\begin{enumerate}

\item Can we characterize "thick" and "thin" diagrams coming from thick and thin FC elements? How does the bijection in Lemma \ref{bijectionthick} translate in terms of diagrams? 

\smallskip

\item What does the duality of Proposition \ref{duality} look like in terms of diagrams?  

\smallskip 

\item Describe the FC elements corresponding to a diagram containing the arrow $(i, i')$, which behave as if they are made of two elements with disjoint supports, that commute. Can they be  fully characterized by conditions on their support,  in addition to the fact that it does not contain $\sigma_i$   ? 

\smallskip

\item Draw consequences of the equivalence relation on diagrams defined above, in particular relative to cellularity properties. 

\smallskip

\item Can we compute the number of equivalence classes, and the cardinality of each? For this would lead to  very interesting partitions of the Catalan number, as follows.  Observe first that once 
$D^+  \cup D^-$ is fixed, what varies in the corresponding equivalence class is nothing but small non-crossing diagrams with exactly the presentation given in ((61)) \cite{Sta15}, hence the cardinality of an equivalence class is a product of  ``small'' Catalan numbers.

 Now let us write $A^{c,p}_n$ as a disjoint union of equivalence classes for $\approx$, say 
$A^{c,p}_n=  \sqcup_{\phi \in \Phi_n^p} \widetilde{ D_\phi}$, where we have chosen a system of representatives $D_\phi, \phi \in \Phi_n^p$, 
for  the equivalence classes in $A^{c,p}_n$, with parameter set $\Phi_n^p$, and we denote by $\widetilde D$ the equivalence class of $D$. Accordingly we find   $$\C_n^p=\sum_{\phi \in \Phi_n^p} \   \sharp  \   \widetilde{ D_\phi} 
\quad \text{ and } \quad C_{n+1} = \sum_{p=0}^n \sum_{\phi \in \Phi_n^p} \   \sharp  \   \widetilde{ D_\phi}, $$
which are respectively a one-parameter partition of the Narayana number and   a two-parameters partition of the Catalan number $C_{n+1}$, in which each term is a product of  smaller Catalan numbers.  

 \smallskip  

\item Rethink the latter questions for the rougher equivalence given by 
$D_1 $ equivalent to $ D_2$ if $D_1^+=D_2^+$. For this equivalence the number of classes is easier to compute, but the cardinality of each class is more complicated than for the previous one.  
\end{enumerate} 

%%\clearpage

\bigskip 

 \section*{Appendix: Generating function}

We want to derive from recurrence relation (\ref{part1cardinalities}) an expression for the generating function of 
  $\C_n^p$,  $0 \le p \le n$, namely $E(x,y)= \sum_{n \ge 0, p\ge 0} \C_n^p x^p y^n.$
We apply the usual procedure, multiplying both sides of  ({part1cardinalities})  by $x^p y^n$ and summing on $n\ge 1$, $p\ge 
0$, and we get (given that the sum in $r$ is $0$ for $p=0$): 
%%\newpage
$$ 
\sum_{n \ge 1, p\ge 0} \!  \C_n^p x^p y^n  \! = \!  \sum_{n \ge 1, p\ge 0} \!   \C_{n-1}^p  x^p y^n +  \sum_{n \ge 1, p\ge 0} \!  \C_{n-1}^{p-1}  x^p y^n   +  
\sum_{n \ge 1, p\ge 1} \! \left( \sum_{r=1}^p  \sum_{i=1}^{n-1}    \C_{n-i-1}^{r-1}     \   \C_{i-1}^{p-r} 
\right)   x^p y^n $$
$$\begin{aligned}   
\text{i.e. }  \ 
E(x,y) -1 &= \   (y+xy) E(x,y)  \quad + \quad x y   \sum_{n \ge 1, p\ge 1}  \left( \sum_{r=1}^p  \sum_{i=1}^{n-1} \C_{n-i-1}^{r-1}     \   \C_{i-1}^{p-r} 
\right)   x^{p-1} y^{n-1} 
\\     
 &= \   (y+xy) E(x,y)  \quad + \quad x y   \sum_{n \ge 0, p\ge 0}  \left( \sum_{r=1}^{p+1}  \sum_{i=1}^{n }  \C_{n-i}^{r-1 }     \   \C_{i-1}^{p+1-r} 
\right)   x^{p} y^{n} 
\\     
 &= \    (y+xy) E(x,y)  \quad + \quad x y   \sum_{n \ge 0, p\ge 0}  \left( \sum_{r=0}^{p}  \sum_{i=0}^{n-1}  \C_{n-i-1}^{r}     \   \C_{i}^{p-r} 
\right)   x^{p} y^{n}  . 
\end{aligned}
$$
Let us write the Cauchy product  for $E(x,y)^2$: 
 
$$\begin{aligned}
E(x,y). E(x,y) &= \left( \sum_{n \ge 0, p\ge 0} \C_n^p x^p y^n \right) \left( \sum_{n \ge 0, p\ge 0} \C_n^p x^p y^n\right)  \\  
&=  \sum_{n \ge 0, p\ge 0} \left( \sum_{ a+i=n, r+d=p} \C_a^r \C_i^d      \right) x^p y^n  
 \\  
&=  \sum_{n \ge 0, p\ge 0} \left(\sum_{   r=0}^p  \sum_{ i=0}^n   \C_{n-i}^r \C_{i}^{p-r}      \right) x^p y^n .
\end{aligned}
$$
We recognise the  coefficient of $x^p y^{n-1}$ as the one we had before and we get eventually: 
$$E(x,y)  -1  =     (y+xy) E(x,y) + xy^2 E(x,y)^2, \text{ best written as: }$$ 
  $$  xy^2 \  E(x,y)^2 + \  (xy+y-1)  \  E(x,y) + 1 =  0 . $$
The function $E(x,y)$ has only non-negative powers of $x$ and $y$ and has  $1$ as the coefficient of 
$x^0 y^0$, so we solve this as:
$$\begin{aligned}
E(x,y)&= \frac{1}{2x y^2} \left[  (1-y-xy) - \left(  (1-y-xy)^2 - 4 x y^2  \right)^{\frac{1}{2}}  \right] \\ 
&=  \frac{(1-y-xy)}{2x y^2} \left[  1 - \left(  1 - 4 x y^2 (1-y-xy)^{-2} \right)^{\frac{1}{2}}  \right]
\end{aligned}
$$
We recall from \cite[(2.5.10)]{Wilf} that 
$$ 
\frac{1}{2z} \left( 1 - \sqrt{1-4z}  \right) 
= \sum_{n \ge 0}    \frac{1}{n+1} \binom{2n}{n}  z^n 
$$
$$\begin{aligned}
\text{hence: } \qquad  \qquad  E(x,y) 
&=   (1-y-xy)^{-1} \  \sum_{t \ge 0} \frac{1}{t+1} \binom{2t}{t}  x^t y^{2t} (1-y-xy)^{-2t}   \\
&=  \sum_{t \ge 0}  \frac{1}{t+1} \binom{2t}{t}  x^t y^{2t} (1-y-xy)^{-2t-1}    
\end{aligned}
$$
Now from \cite[(2.5.7)]{Wilf} we have   
$$(1-y-xy)^{-2t-1} =   \sum_{k \ge 0} \binom{k+2t}{k}  (1+x)^k  y^k
$$
that we may plug in the previous expression whenever convenient. In particular, we can fix $n$ and develop 
the one-variable generating function $F_n(x) = \sum_{p \ge 0} \C_n^p x^p$, which is the coefficient of $y^n$ 
in the above series. We obtain: 
$$\begin{aligned}
F_n(x)&=   \sum_{\substack{
   t \ge 0 \\
 k \ge 0 \\
k+2t=n
  }}     \frac{1}{t+1} \binom{2t}{t}  x^t  \binom{k+2t}{k}  (1+x)^k  \\
&=   \sum_{\substack{
   t \ge 0 \\
 2t\le    n
  }}     \frac{1}{t+1} \binom{2t}{t}  x^t  \binom{n}{n-2t}  (1+x)^{n-2t} 
  \\
&=   \sum_{\substack{
   t \ge 0 \\
 2t\le    n
  }}  \sum_{i=0}^{n-2t}       \frac{1}{t+1} \binom{2t}{t}   \binom{n}{n-2t}  \binom{n-2t}{i} x^{t+i} 
 \\
&=   \sum_{\substack{
   t \ge 0 \\
 2t\le    n
  }}  \sum_{i=0}^{n-2t}       \frac{1}{t+1} \frac{(2t)! \ n! \ (n-2t)!}{t! \ t! \  (n-2t)! \ (2t)! \ i! \ (n-2t-i)!} \   x^{t+i} 
\\
&=     \sum_{\substack{
   t \ge 0 \\
 2t\le    n
  }}  \sum_{i=0}^{n-2t}       \frac{1}{t+1} \frac{  \ n!  }{t! \ t! \    \ i! \ (n-2t-i)!} \   x^{t+i} \end{aligned}
$$ 
We change variables as follows: 

 $(t,i) \mapsto (p,t) \ (t+i = p,  0 \le p \le n ; \   2t+i=p+t , 0 \le p+t \le n  )   $  

\noindent 
and get: 
$$\begin{aligned}
F_n(x)   &=     \sum_{p=0}^{n}      \left(     \sum_{t=0}^{p}     \frac{1}{t+1} \frac{  \ n!  }{t! \ t! \    \ (p-t)! \ (n-p-t)!} \right)   x^{p} 
\\
&=     \sum_{p=0}^{n}      \left(     \sum_{t=0}^{p}     \frac{1}{t+1} \frac{ p! \ (n-p)! \ n!  }{t! \ t! \  p! \ (n-p)!  \ (p-t)! \ (n-p-t)!} \right)   x^{p} 
\\
&=     \sum_{p=0}^{n}   \binom{n}{p}    \left(     \sum_{t=0}^{p}     \frac{1}{t+1}  \binom{p}{t} 
 \binom{n-p}{t} 
   \right)   x^{p} 
\end{aligned}
$$
 
The following lemma may be known yet we only found a simpler version in \cite{Sta12} so we include the proof. 
\begin{lemma}\label{binomiallemma}
$$
     \sum_{t=0}^{p}     \frac{1}{t+1}  \binom{p}{ t} 
 \binom{n-p}{t} = \frac{1}{p+1}  \binom{n+1}{n-p+1} = \frac{1}{p+1}  \binom{n+1}{p}  . 
$$
\end{lemma}
\begin{proof}
The simpler version  in \cite[Example 1.1.17]{Sta12} gives   the coefficient of $x^p$ in  $(1+x)^n$ written as 
  $(1+x)^p   (1+x)^{n-p}  $. We have 
 $\sum_{t=0}^{n-p}    
 \binom{n-p}{t} x^t =  (1+x)^{n-p}$ and by integrating term by term $ (1+x)^{ p} $   we also have 
$$  \sum_{t=0}^{p}     \frac{1}{t+1}  \binom{p}{t} x^{t+1}  =    \int_0^x (1+u)^p   du    = 
\frac{1}{p+1} \left[ (1+x)^{p+1} - 1 \right]  .  
$$    
 Hence 
$   \sum_{t=0}^{p}     \frac{1}{t+1}  \binom{p}{ t} 
 \binom{n-p}{t} =  \sum_{t=0}^{p}     \frac{1}{t+1}  \binom{p}{  t} 
 \binom{n-p}{n-p-t} 
$ is the coefficient of $x^{n-p+1} $ in 
$$ \frac{1}{p+1} \left[ (1+x)^{p+1} - 1 \right] (1+x)^{n-p} =\frac{1}{p+1} \left[ (1+x)^{n+1} - (1+x)^{n-p} \right]  
$$ 
namely  $ \frac{1}{p+1}  \binom{n+1}{n-p+1} $. 
\end{proof}  
 From this we get Theorem \ref{Narayana} as announced.

\bigskip
 
\renewcommand{\refname}{REFERENCES}

	\end{document}